\documentclass[11pt]{amsart}
\usepackage[latin1]{inputenc}
\usepackage[all]{xy}
\usepackage{tikz-cd}
\usepackage{pdflscape}
\usepackage{graphics}
\usepackage{amscd}
\usepackage{amssymb}
\usepackage{amsthm}
\usepackage{enumitem}
\usepackage[left=0.0cm,right=0.0cm,top=2.7cm,bottom=3cm]{geometry}
\usepackage{comment}
\newtheorem{thm}{Theorem}[section]

\newtheorem{pro}[thm]{Proposition}
\newtheorem{lem}[thm]{Lemma}
\newtheorem{cor}[thm]{Corollary}

\newtheorem*{rem*}{Remarks}
\newtheorem{rems}[thm]{Remark}
\newtheorem*{conj*}{Conjecture}

\DeclareMathOperator{\Q}{\mathbb{Q}}

\DeclareMathOperator{\Z}{\mathbb{Z}}

\DeclareMathOperator{\Spec}{Spec}

\DeclareMathOperator{\id}{id}

\DeclareMathOperator{\res}{res}

\newcommand{\mrm}{\mathrm}
\newcommand{\mbb}{\mathbb}

\title{Kato explicit reciprocity law for Siegel modular forms of weight $(3, 3)$}
\author{Francesco Lemma}
\address{Universit\'e Paris Cit\'e and Sorbonne Universit\'e, CNRS, Institut math\'ematique de Jussieu-Paris Rive Gauche, UMR 7586, B\^atiment Sophie Germain, Case 7012, 75205 Paris Cedex 13, France.}
\email{francesco.lemma@imj-prg.fr}
\author{Tadashi Ochiai}
\address{Department of Mathematics, Institute of Science Tokyo, 2-12-1 Ookayama, Meguro-ku, Tokyo 152-8551, Japan}
\email{ochiai@math.titech.ac.jp}
\thanks{This work was supported by JSPS KAKENHI (Grant-in-Aid for Scientific Research (B): Grant Number 23K25763, 
Grant-in-Aid for Challenging Exploratory Research: Grant Number 23K17651}
\begin{document}

\begin{abstract} 
We extend Kato explicit reciprocity law, in the version written by Scholl, for a modular curve to a product of two modular curves. By embedding the product of two modular curves in the Siegel threefold, we deduce an explicit reciprocity law for the unique critical twist of the $p$-adic Galois representation attached to cuspidal Siegel modular forms of weight $(3,3)$.
\end{abstract}

\maketitle

\tableofcontents

\section{Motivation and statement of the main result} \label{intro}

In the papers \cite{kato99} and \cite{kato04}, Kato constructed an Euler system related to critical values of $L$-functions of modular forms. By using this Euler system, he bounded Selmer groups attached to modular forms and proved one of the divisibilities of the Iwasawa main conjecture for modular forms. The main ingredient in the work of Kato is the so called generalized explicit reciprocity law, which is the computation of the image by Bloch-Kato dual exponential map of the Euler system in terms of the the product of two Eisenstein series. Via the method of Rankin-Selberg, this allows Kato to relate the image of the Euler system by the dual exponential map to critical values of $L$-functions of modular forms of weight $k \geq 2$. In the particular case of modular forms of weight $2$ with good reduction at the prime $p$, Kato's ideas were exposed by Scholl in the paper \cite{scholl}. In this paper, we extend the main result \cite[Corollary 3.2.4]{scholl} of the article \cite{scholl} for a modular curve to a product of two modular curves. By embedding the product of two modular curves in the Siegel threefold, we deduce an explicit reciprocity law for the  unique critical twist of the $p$-adic Galois representation attached to Siegel modular forms of weight $(3,3)$, which is the minimal cohomological weight.\\

To state the main result of this article, we need to introduce some notation. Let $I_2$ be the identity matrix of size two and let $\psi$ be the $4$ by $4$ matrix given by
$$
\psi=
\begin{pmatrix}
 & I_2\\
-I_2 & \\
\end{pmatrix}.
$$
The symplectic group $G=\mathrm{GSp}(4)$ is defined as
$$
G=\{g \in \mathrm{GL}(4) \,|\, ^tg \psi g = \nu(g) \psi, \, \nu(g) \in \mathrm{GL}(1)\}.
$$
It is a reductive linear algebraic group over $\mbb{Q}$ and contains the group $H=\mathrm{GL}(2) \times_{\mathrm{GL}(1)} \mathrm{GL}(2)$, where the fiber product is over the determinant, via the embedding
$\iota: H \hookrightarrow G$ 
defined by
$$
\iota\left( \begin{pmatrix}
a & b\\
c & d\\
\end{pmatrix}, \begin{pmatrix}
a' & b'\\
c' & d'\\
\end{pmatrix} \right) =
\begin{pmatrix}
a & & b & \\
 & a' &  & b'\\
c &  & d & \\
 & c' &  & d'\\
\end{pmatrix}.
$$
Let $p>2$ denote a fixed prime number and let $N \geq 3$ denote an integer prime to $p$. Let $\mathfrak{o}$ denote the $\Z_p$-algebra $\Z_p[\zeta_N]$ and $K$ denote its fraction field. We let $\mathcal{Y}_H(N)/\mathfrak{o}$, resp. $\mathcal{Y}_G(N)/\mathfrak{o}$, denote the canonical integral model over $\mathfrak{o}$ of the Shimura variety of $H$, resp. of $G$, at principal level $N$. We have $\mathcal{Y}_H(N)=\mathcal{Y}(N) \times_\mathfrak{o} \mathcal{Y}(N)$, where $\mathcal{Y}(N)$ is the canonical integral model over $\mathfrak{o}$ of the modular curve of principal level $N$. Then $\mathcal{Y}_H(N)/\mathfrak{o}$, resp. $\mathcal{Y}_G(N)/\mathfrak{o}$, is smooth of relative dimension $2$, resp. $3$. By a slight abuse of notation, we denote again by $\iota: \mathcal{Y}_H(N) \hookrightarrow \mathcal{Y}_G(N)$ the closed embedding over $\mathfrak{o}$ induced by the embedding of groups $\iota$. For $n \geq 1$, we let $\mathfrak{o}_n$ denote the discrete valuation ring $\mathbb{Z}_p[\zeta_{Np^n}]$ and let $K_n$ denote its fraction field.\\

Let us recall that for any continuous $p$-adic representation $V$ of the absolute Galois group of a finite extension of $L/\mathbb{Q}_p$, Bloch-Kato \cite{bloch-kato} defined the dual exponential map
$$
\mathrm{exp}^*: H^1\left(L_m, V \right) \rightarrow L_m \otimes_{L} \mrm{Fil}^0 D_{dR}\left( V \right)
$$
for any finite extension $L_m/L$. In what follows, we take $V=H^3_{\text{\'et}}(\mathcal{Y}_G(N) \otimes_{\mathfrak{o}} \overline{K}, \Q_p(2))$, which is a continuous $p$-adic representation of the absolute Galois group of $K$. Noting that, by \cite{beilinson-padic}, we have a canonical isomorphism of $\mathbb{Q}_p$-vector spaces
$$
\mrm{Fil}^0 D_{dR}\left( H^3_{\text{\'et}}(\mathcal{Y}_G(N) \otimes_{\mathfrak{o}} \overline{K}, \Q_p(2)) \right) \simeq K \otimes_{\mathfrak{o}} \mrm{Fil}^2 H^3_{dR}(\mathcal{Y}_G(N)),
$$
the map $\mathrm{exp^*}$ can be regarded as a map
$$
\mathrm{exp}^*: H^1\left(K_m, H^3_{\text{\'et}}(\mathcal{Y}_G(N) \otimes_{\mathfrak{o}} \overline{K}, \Q_p(2)) \right) \rightarrow K_m \otimes_{\mathfrak{o}} \mrm{Fil}^2 H^3_{dR}(\mathcal{Y}_G(N)).
$$
This is the lower right hand vertical map in the diagram of the following theorem, whose commutativity is the main result of this article. For a precise definition of the other spaces and morphisms composing this diagram, we refer the reader to section \ref{main_cc}. 

\begin{thm} \label{main} Let $m \geq 1$ be an integer. The diagram 
$$
\begin{tikzcd}
\underleftarrow{\lim}_n K_2\left(\mathcal{Y}_H(N)\otimes_{\mathfrak{o}} \mathfrak{o}_n \right) \otimes \mu_{p^n}^{-1}  \ar[r, "\iota_* \circ (Tr_{n,m} \circ Ch_n)_{n \geq 1}"] \ar[d, "d\log" left]& H^4_{\text{\'et}} \left( \mathcal{Y}_G(N) \otimes_{\mathfrak{o}} K_m, \Z_p(2) \right)^0 \ar[d, "Hochschild-Serre"]\\
\underleftarrow{\lim}_n H^0\left(\overline{\mathcal{Y}_H(N)} \otimes_{\mathfrak{o}} \mathfrak{o}_n, \Omega^2_{\overline{\mathcal{Y}_H(N)} \otimes \mathfrak{o}_n}(\log) \right)(-1) \ar[d, "pr" left] & H^1\left(K_m, H^3_{\text{\'et}}(\mathcal{Y}_G(N) \otimes_{\mathfrak{o}} \overline{K}, \Z_p(2)) \right) \ar[d]\\
\underleftarrow{\lim}_n \left( \Omega^1_{\mathfrak{o}_n/\mathfrak{o}}(-1) \otimes_{\mathfrak{o}} \mrm{Fil}^1 H^1_{dR}(\mathcal{Y}_H(N)) \right) \ar[d, "d \log \zeta_{p^n} \otimes \langle \zeta_{p^n} \rangle^{-1} \mapsto 1" left] & H^1\left(K_m, H^3_{\text{\'et}}(\mathcal{Y}_G(N) \otimes_{\mathfrak{o}} \overline{K}, \Q_p(2)) \right) \ar[d, "\mrm{exp}^*"]\\
\underleftarrow{\lim}_n \left( \mathfrak{o}_n/\mathfrak{d}_n \otimes_{\mathfrak{o}} \mrm{Fil}^1 H^1_{dR}(\mathcal{Y}_H(N)) \right) \ar[r, "\iota_* \circ (1/p^n tr_{n,m})_{n \geq m}"] & K_m \otimes_{\mathfrak{o}} \mrm{Fil}^2 H^3_{dR}(\mathcal{Y}_G(N)).
\end{tikzcd}
$$
is commutative.
\end{thm}

Let $\pi$ be the automorphic representation of $G(\mathbb{A})$ corresponding to a cuspidal Siegel modular form of weight $(3,3)$ which is not endoscopic nor CAP. Let $\pi_f$ denote the non-archimedean component of $\pi$ and let $E$ denote the field of rationality of $\pi_f$, which is a number field. Let $K_G(N) \subset G(\widehat{\mathbb{Z}})$ denote the principal congruence subgroup of level $N$ and let $Y_G(N)/\mathbb{Q}$ denote the Shimura vairety of $G$ at level $K_G(N)$. The $\pi_f^{K_G(N)}$-isotypical component of $H^3_{\text{\'et}}\left(Y_G(N) \otimes_{\mathbb{Q}} \overline{\mathbb{Q}}, \mathbb{Q}_p \right) \otimes_{\mathbb{Q}} E$ is $V_\pi \otimes_{\mathbb{Q}}\pi_f^{K_G(N)}$ where $V_\pi$ is the four dimensional $p$-adic Galois representation attached to $\pi$ \cite{weissauer}. By choosing a vector $\psi_f \in \check{\pi}_f^{K_G(N)}$, where $\check{\pi}_f$ denotes the contragredient representation of $G(\mathbb{A}_f)$, we obtain the Galois equivariant map 
$$
H^3_{\text{\'et}}\left(Y_G(N) \otimes_{\mathbb{Q}} \overline{\mathbb{Q}}, \mathbb{Q}_p \right) \otimes_{\mathbb{Q}} E \rightarrow V_\pi.
$$
Hence, as a direct consequence of Theorem \ref{main}, we obtain an explicit reciprocity law for the twist $V_\pi(2)$ of $V_\pi$.

\begin{cor} \label{main-cor} Let $m \geq 1$ be an integer. The diagram 
$$
\begin{tikzcd}
\underleftarrow{\lim}_n K_2\left(\mathcal{Y}_H(N)\otimes_{\mathfrak{o}} \mathfrak{o}_n \right) \otimes \mu_{p^n}^{-1}  \ar[r, "\iota_* \circ (Tr_{n,m} \circ Ch_n)_{n \geq 1}"] \ar[d, "d\log" left]& H^4_{\text{\'et}} \left( \mathcal{Y}_G(N) \otimes_{\mathfrak{o}} K_m, \Z_p(2) \right)^0 \ar[d, "Hochschild-Serre"]\\
\underleftarrow{\lim}_n H^0\left(\overline{\mathcal{Y}_H(N)} \otimes_{\mathfrak{o}} \mathfrak{o}_n, \Omega^2_{\overline{\mathcal{Y}_H(N)} \otimes \mathfrak{o}_n}(\log) \right)(-1) \ar[d, "pr" left] & H^1\left(K_m, H^3_{\text{\'et}}(\mathcal{Y}_G(N) \otimes_{\mathfrak{o}} \overline{K}, \Z_p(2)) \right) \ar[d]\\
\underleftarrow{\lim}_n \left( \Omega^1_{\mathfrak{o}_n/\mathfrak{o}}(-1) \otimes_{\mathfrak{o}} \mrm{Fil}^1 H^1_{dR}(\mathcal{Y}_H(N)) \right) \ar[d, "d \log \zeta_{p^n} \otimes \langle \zeta_{p^n} \rangle^{-1} \mapsto 1" left] & H^1\left(K_m, V_\pi(2)) \right) \ar[d, "\mrm{exp}^*"]\\
\underleftarrow{\lim}_n \left( \mathfrak{o}_n/\mathfrak{d}_n \otimes_{\mathfrak{o}} \mrm{Fil}^1 H^1_{dR}(\mathcal{Y}_H(N)) \right) \ar[r, "\iota_* \circ (1/p^n tr_{n,m})_{n \geq m}"] & K_m \otimes_{K} \mrm{Fil}^2 D_{dR}(V_\pi)
\end{tikzcd}
$$
is commutative.
\end{cor}

\begin{rems} We know that the Hodge decomposition of the $\pi_f^{K_G(N)}$-isotypical component of the Betti cohomology $H^3_{B}\left(Y_G(N)(\mathbb{C}), \mathbb{Q} \right) \otimes_{\mathbb{Q}} E$ has Hodge types $(3,0), (2,1), (1,2), (0,3)$. It follows that the unique critical twist of $V_\pi$, in the sense of \cite{deligne}, is $V_\pi(2)$. As a consequence Corollary \ref{main-cor} is an explicit reciprocity law for the unique critical twist of $V_\pi$.
\end{rems}

\textbf{Acknowledgements.} It is a pleasure to thank Antonio Cauchi, Pierre Colmez and Joaquin Rodrigues Jacinto for useful remarks on this work. Parts of this work were written while the first named author was visiting the Institute of Science Tokyo, which he would like to thank for providing excellent work conditions.
 Parts of this work were written while the second named author was visiting Universit\'e Paris Cit\'e and Sorbonne Universit\'e. He would like to thank for their hospitality.
\section{The explicit reciprocity law}

\subsection{Some commutative diagrams} \label{main_cc}To state our theorem, we introduce the following notation. For any integer $n \geq 0$, let $\mathfrak{o}_n$ denote the discrete valuation ring $\Z_p[\zeta_{Np^n}]$ and let $\mathfrak{o}$ denote $\mathfrak{o}_{0}$. Let $K_n$, resp. $K$, denote the fraction field of $\mathfrak{o}_n$, resp. $\mathfrak{o}$. Let $\mathcal{Y}_H(N)=\mathcal{Y}(N) \times_{\mathfrak{o}} \mathcal{Y}(N)$ and let $\overline{\mathcal{Y}_H(N)}$ denote the canonical smooth compactification of $\mathcal{Y}_H(N)$ such that $\overline{\mathcal{Y}_H(N)}-\mathcal{Y}_H(N)$ is a divisor with relatively normal crossings. In what follows we define the homomorphisms composing the diagrams that we want to state.\\

We start by introducing the maps composing the diagram which we want to prove to be commutative.

\subsubsection{} First we have the Chern character 
$$
K_2\left(\mathcal{Y}_{H}(N) \otimes_{\mathfrak{o}} \mathfrak{o}_n \right) \rightarrow H^2_{\text{et}}(\mathcal{Y}_H(N) \otimes_{\mathfrak{o}} K_n, \Z/p^n\Z(2)),
$$
which we compose with the cup-product
$$
H^2_{\text{et}}(\mathcal{Y}_H(N) \otimes_{\mathfrak{o}} K_n, \Z/p^n\Z(2)) \otimes \mu_{p^n}^{-1} \rightarrow H^2_{\text{et}}(\mathcal{Y}_H(N) \otimes_{\mathfrak{o}} K_n, \Z/p^n\Z(1))
$$
to obtain the map
$$
Ch_n: K_2\left(\mathcal{Y}_{H}(N) \otimes_{\mathfrak{o}} \mathfrak{o}_n \right) \otimes \mu_{p^n}^{-1} \rightarrow H^2_{\text{et}}(\mathcal{Y}_H(N) \otimes_{\mathfrak{o}} K_n, \Z/p^n\Z(1)).
$$

\subsubsection{} \label{K20} Let $H^2_{\text{et}}(\mathcal{Y}_H(N) \otimes_{\mathfrak{o}} K_n, \Z/p^n\Z(1))^0 \subset H^2_{\text{et}}(\mathcal{Y}_H(N) \otimes_{\mathfrak{o}} K_n, \Z/p^n\Z(1))$ denote the classes which are homologically trivial, which are defined as
$$
H^2_{\text{et}}(\mathcal{Y}_H(N) \otimes_{\mathfrak{o}} K_n, \Z/p^n\Z(1))^0= \ker ( H^2_{\text{et}}(\mathcal{Y}_H(N) \otimes_{\mathfrak{o}} K_n, \Z/p^n\Z(1)) 
$$
$$
\rightarrow H^2_{\text{et}}(\mathcal{Y}_H(N) \otimes_{\mathfrak{o}} \overline{K}, \Z/p^n\Z(1))).
$$
Similarly, we denote by 
$$
\left( K_2\left(\mathcal{Y}_{H}(N) \otimes_{\mathfrak{o}} \mathfrak{o}_n \right) \otimes \mu_{p^n}^{-1} \right)^0=Ch_n^{-1}\left( H^2_{\text{et}}(\mathcal{Y}_H(N) \otimes_{\mathfrak{o}} K_n, \Z/p^n\Z(1))^0 \right).
$$
The first edge homomorphism of the Hochschild-Serre spectral sequence is a natural map
$$
H^2_{\text{et}}(\mathcal{Y}_H(N) \otimes_{\mathfrak{o}} K_n, \Z/p^n\Z(1))^0 \rightarrow H^1(K_n, H^1_{\text{et}}(\mathcal{Y}_H(N) \otimes_{\mathfrak{o}} \overline{K}, \Z/p^n\Z(1))).
$$

\subsubsection{} Let $\pi_1 \pmod{p^n}:  H^1_{\text{et}}(\mathcal{Y}_H(N) \otimes_{\mathfrak{o}} \overline{K}, \Z/p^n\Z(1)) \rightarrow \overline{\mathfrak{o}}/p^{n} \otimes_{\mathfrak{o}} H^0\left( \overline{\mathcal{Y}_H(N)},  \Omega^1_{\overline{\mathcal{Y}_H(N)}}(\mathrm{log}) \right)$ be the map defined in \cite{scholl} p. 408 for $Y=\mathcal{Y}_H(N)$. By a slight abuse of notation, we will still denote by $\pi_1 \pmod{p^n}$ the map induced by $\pi_1 \pmod{p^n}$ after applying the functor $M \mapsto H^1(K_n, M)$. 

\subsubsection{}  The logarithm of the cyclotomic character is a continuous homomorphism $$\log \chi_{cyc} \in \mrm{Hom}(\mrm{Gal}(\overline{K}/\Q_p), \Z_p)=H^1(\Q_p, \Z_p).$$
As $K_n=\Q_p(\zeta_{Np^n})$ contains the $p^n$-th roots of unity, for any $\sigma \in \mrm{Gal}(\overline{K}/K_n)$ we have $\chi_{cyc}(\sigma) \equiv 1 \pmod{p^n}$ and so $\log \chi_{cyc}(\sigma) \equiv 0 \pmod{p^n}$. As a consequence, the cup-product with $1/p^{n-1}\log \chi_{cyc}$ is a homomorphism
$
\cup (1/p^{n-1}\log \chi_{cyc}): \mathfrak{o}_n=H^0(K_n, \widehat{\overline{\mathfrak{o}}}) \rightarrow H^1(K_n, \widehat{\overline{\mathfrak{o}}}).
$
By a slight abuse of notation, we still denote by
\begin{equation*} \label{cup-log-cyclo1}
\cup (1/p^{n-1}\log \chi_{cyc}): \mathfrak{o}_n/p^{n-1} \rightarrow H^1(K_n, \overline{\mathfrak{o}}/p^{n-1})
\end{equation*}
the homomorphism induced by the one above by reduction modulo $p^{n-1}$.

\subsubsection{} We let
$$
d\log: K_2(\mathcal{Y}_H(N) \otimes_{\mathfrak{o}} \mathfrak{o}_n) \otimes \mu_{p^n}^{-1} \rightarrow H^0 \left( \overline{\mathcal{Y}_H(N)} \otimes_{\mathfrak{o}} \mathfrak{o}_n, \Omega^2_{\overline{\mathcal{Y}_H(N)}\otimes_{\mathfrak{o}} \mathfrak{o}_n}(\log) \right)(-1)
$$
be the homomorphism defined in at p. 393 of \cite{scholl} (see also Remark (iv) p. 411 of \cite{scholl}). We note that this homomorphism satisfies $d\log\{u,v\}=d\log(u)\wedge d\log(v)$ for $u, v \in \mathcal{O}(\mathcal{Y}_H(N) \otimes \mathfrak{o}_n)^\times$.

\subsubsection{} Finally, we let 
$$
pr: H^0 \left( \overline{\mathcal{Y}_H(N)} \otimes_{\mathfrak{o}} \mathfrak{o}_n, \Omega^2_{\overline{\mathcal{Y}_H(N)}\otimes_{\mathfrak{o}} \mathfrak{o}_n}(\log) \right)(-1) \rightarrow \Omega^1_{\mathfrak{o}_n/\mathfrak{o}}(-1) \otimes_{\mathfrak{o}} \mrm{Fil}^1H^1_{dR}(\mathcal{Y}_H(N))
$$
be defined as follows. We have $\mrm{Fil}^1H^1_{dR}(\mathcal{Y}_H(N))=H^0(\overline{\mathcal{Y}_H(N)}, \Omega^1_{\overline{\mathcal{Y}_H(N)}}(\log))$ and we have a canonical isomorphism of sheaves of $\mathfrak{o}_n \otimes \mathcal{O}_{\overline{\mathcal{Y}_H(N)}}$-modules
$$
\Omega^1_{\overline{\mathcal{Y}_H(N)}\otimes_{\mathfrak{o}} \mathfrak{o}_n}(\log)) \simeq (\Omega^1_{\mathfrak{o}_n/\mathfrak{o}} \otimes_{\mathfrak{o}} \mathcal{O}_{\overline{\mathcal{Y}_H(N)}}) \oplus \left( \mathfrak{o}_n \otimes_{\mathfrak{o}} \Omega^1_{\overline{\mathcal{Y}_H(N)}}(\log) \right)
$$ as $\overline{\mathcal{Y}_H(N)}/\mathfrak{o}$ is smooth. By taking the second exterior power, we obtain a canonical projection
$
\Omega^2_{\overline{\mathcal{Y}_H(N)}\otimes \mathfrak{o}_n}(\log) \rightarrow (\Omega^1_{\mathfrak{o}_n/\mathfrak{o}} \otimes_{\mathfrak{o}} \mathcal{O}_{\overline{\mathcal{Y}_H(N)}}) \otimes_{\mathcal{O}_{\overline{\mathcal{Y}_H(N)}}} \Omega^1_{\overline{\mathcal{Y}_H(N)}}(\log)
$
which induces the homomorphism $pr$ on global sections.

\begin{thm} \label{cdfinite} There exists an integer $c \geq 0$ such that the diagram
$$
\begin{tikzcd}
\left( K_2\left(\mathcal{Y}_{H}(N) \otimes_{\mathfrak{o}} \mathfrak{o}_n \right) \otimes \mu_{p^n}^{-1} \right)^0 \ar[r, "Ch_n"] \ar[d, "d\log" left]& H^2_{\text{et}}(\mathcal{Y}_H(N) \otimes_{\mathfrak{o}} K_n, \Z/p^n\Z(1))^0 \ar[d, "Hochschild-Serre"]\\
H^0\left(\overline{\mathcal{Y}_H(N)} \otimes_{\mathfrak{o}} \mathfrak{o}_n, \Omega^2_{\overline{\mathcal{Y}_H(N)} \otimes_{\mathfrak{o}} \mathfrak{o}_n}(\mathrm{log})\right)(-1)  \ar[d, "pr" left]&  H^1(K_n, H^1_{\text{et}}(\mathcal{Y}_H(N) \otimes_{\mathfrak{o}} \overline{K}, \Z/p^n\Z(1))) \ar[d, "\pi_1 \pmod{p^{n-1}}"]\\
\Omega^1_{\mathfrak{o}_n/\mathfrak{o}}(-1) \otimes \mathrm{Fil}^1H^1_{dR}(\mathcal{Y}_H(N)/\mathfrak{o}) \ar[d, "d \log \zeta_{p^n} \otimes \langle \zeta_{p^n} \rangle^{-1} \mapsto 1" left] &  H^1(K_n, \overline{\mathfrak{o}}/p^{n-1}) \otimes_{\mathfrak{o}} \mathrm{Fil}^1H^1_{dR}(\mathcal{Y}_H(N)/\mathfrak{o}) \\
\mathfrak{o}_n/\mathfrak{d}_n \otimes_{\mathfrak{o}} \mathrm{Fil}^1H^1_{dR}(\mathcal{Y}_H(N)/\mathfrak{o})  \ar[r]& \mathfrak{o}_n/p^{n-1} \otimes_{\mathfrak{o}} \mathrm{Fil}^1H^1_{dR}(\mathcal{Y}_H(N)/\mathfrak{o}) \ar[u, "\cup \frac{1}{p^n}\log \chi_{cyc}" right]
\end{tikzcd}
$$
commutes up to $p^c$-torsion.
\end{thm}

To state a first Corollary of Theorem \ref{cdfinite} we need to explain the compatibilities with respect to the traces of the maps composing the diagram of Theorem \ref{cdfinite}.

\subsubsection{} According to \cite{scholl} p. 393 (iv), the map $d\log$ is compatible with the trace maps on the source and on the target and hence induces a map 
\begin{equation} \label{dlog1}
d\log: \underleftarrow{\lim}_n K_2(\mathcal{Y}_H(N) \otimes_{\mathfrak{o}} \mathfrak{o}_n) \otimes \mu_{p^n}^{-1} \rightarrow  \underleftarrow{\lim}_n H^0 \left( \overline{\mathcal{Y}_H(N)} \otimes_{\mathfrak{o}} \mathfrak{o}_n, \Omega^2_{\overline{\mathcal{Y}_H(N)}\otimes \mathfrak{o}_n}(\log) \right)(-1)
\end{equation}
denoted in the same way by a slight abuse of notation.

\subsubsection{} The relative different $\mathfrak{d}_n$ of the extension $\mathfrak{o}_n/\mathfrak{o}$ is an ideal of $\mathfrak{o}_n$ such that the homomorphism of $\mathfrak{o}_n$-modules $\mathfrak{o}_n(1) \rightarrow \Omega^1_{\mathfrak{o}_n/\mathfrak{o}}$ defined by $1 \otimes \zeta_{p^n} \mapsto d\log \zeta_{p^n}$ induces an isomorphism 
$
\mathfrak{o}_n/\mathfrak{d}_n(1) \simeq \Omega^1_{\mathfrak{o}_n/\mathfrak{o}}.
$
According to \cite[Proposition 3.3.12]{scholl}, the diagram
$$
\begin{tikzcd}
\mathfrak{o}_n/\mathfrak{d}_n(1) \ar[r, "\sim"] \ar[d, "t_{n,m}"]& \Omega^1_{\mathfrak{o}_n/\mathfrak{o}} \ar[d, "tr_{n,m}"]\\
\mathfrak{o}_m/\mathfrak{d}_m(1) \ar[r, "\sim"] & \Omega^1_{\mathfrak{o}_m/\mathfrak{o}}
\end{tikzcd}
$$
where $t_{n,m}=\frac{1}{p^{n-m}}tr_{n,m}$ and $tr_{n,m}$ is the trace map is commutative. As a consequence, we obtain the map 
\begin{equation} \label{different1}
\underleftarrow{\lim}_n \left( \Omega^1_{\mathfrak{o}_n/\mathfrak{o}}(-1) \otimes_{\mathfrak{o}} \mrm{Fil}^1 H^1_{dR}(\mathcal{Y}_H(N)) \right) \rightarrow \underleftarrow{\lim}_n \left( \mathfrak{o}_n/\mathfrak{d}_n \otimes_{\mathfrak{o}} \mrm{Fil}^1 H^1_{dR}(\mathcal{Y}_H(N)) \right),
\end{equation}
which is an isomorphism.

\subsubsection{} 

\begin{lem} \label{different-generator} For any integer $n \geq 1$, the relative different $\mathfrak{d}_n$ of $\mathfrak{o}_n/\mathfrak{o}$ is the ideal of $\mathfrak{o}_n$ generated by $p^n(\zeta_p-1)^{-1}$.
\end{lem}

\begin{proof} As $N$ is prime to $p$, we have $\mathfrak{o}_n=\mathfrak{o}[\zeta_{p^n}]$ and the minimal polynomial of $\zeta_{p^n}$ over $\mathfrak{o}$ is $f(X)=\frac{X^{p^n}-1}{X^{p^{n-1}}-1}$. Hence, according to \cite[\S6 Corollaire 2]{serre}, the ideal $\mathfrak{d}_n$ is generated by $f'(\zeta_{p^n})=p^n\zeta_p(\zeta_p-1)^{-1}$. This implies the statement because $\zeta_p \in \mathfrak{o}_n^\times$.
\end{proof}

For any integers $n \geq m \geq 1$ we have the trace homomorphisms $\mathfrak{o}_n \rightarrow \mathfrak{o}_m$ which induce the homomorphism
$
tr_{n,m}: \mathfrak{o}_n/\mathfrak{d}_n = \mathfrak{o}_n/p^n(\zeta_p-1)^{-1} \rightarrow \mathfrak{o}_m/p^n(\zeta_p-1)^{-1}.
$
over $\mathfrak{o}$ with $\mrm{Fil}^1 H^1_{dR}(\mathcal{Y}_H(N))$ and taking the inverse limit we obtain
\begin{equation} \label{trace}
\left( \frac{1}{p^{n}} tr_{n,m}\right)_{n \geq m}: \underleftarrow{\lim}_n \left( \mathfrak{o}_n/\mathfrak{d}_n \otimes_{\mathfrak{o}} \mrm{Fil}^1 H^1_{dR}(\mathcal{Y}_H(N))\right) \rightarrow K_m \otimes_{\mathfrak{o}} \mrm{Fil}^1 H^1_{dR}(\mathcal{Y}_H(N)).
\end{equation}

By taking the inverse limits over $n$ in the diagram of Theorem \ref{cdfinite} and composing with the trace to $K_m$ we obtain the following, where we denote $\pi_1=\underleftarrow{\lim}_n\left( \pi_1 \pmod{p^n} \right)$.

\begin{cor} \label{corollary11} There exists an integer $c \geq 0$ such that the diagram
$$
\begin{tikzcd}
\underleftarrow{\lim}_n \left( K_2\left(\mathcal{Y}_{H}(N) \otimes_{\mathfrak{o}} \mathfrak{o}_n \right) \otimes \mu_{p^n}^{-1} \right)^0 \ar[r, "(Tr_{m,n} \circ Ch_n)_{n \geq 1}"] \ar[d, "d\log" left]& H^2_{\text{et}}(\mathcal{Y}_H(N) \otimes_{\mathfrak{o}} K_m, \Z_p(1))^0 \ar[d, "Hochschild-Serre"]\\
\underleftarrow{\lim}_n H^0\left(\overline{\mathcal{Y}_H(N)} \otimes_{\mathfrak{o}} \mathfrak{o}_n, \Omega^2_{\overline{\mathcal{Y}_H(N)} \otimes_{\mathfrak{o}} \mathfrak{o}_n}(\mathrm{log})\right)(-1)  \ar[d, "pr" left]&  H^1(K_m, H^1_{\text{et}}(\mathcal{Y}_H(N) \otimes_{\mathfrak{o}} \overline{K}, \Z_p(1))) \ar[d, "\pi_1"]\\
\underleftarrow{\lim}_n \left( \Omega^1_{\mathfrak{o}_n/\mathfrak{o}}(-1) \otimes_{\mathfrak{o}} \mathrm{Fil}^1H^1_{dR}(\mathcal{Y}_H(N)/\mathfrak{o}) \right) \ar[d, "d \log \zeta_{p^n} \otimes \langle \zeta_{p^n} \rangle^{-1} \mapsto 1" left] &  H^1\left(K_m, \widehat{\overline{\mathfrak{o}}} \right) \otimes_{\mathfrak{o}} \mathrm{Fil}^1H^1_{dR}(\mathcal{Y}_H(N)/\mathfrak{o}) \\
\underleftarrow{\lim}_n \left( \mathfrak{o}_n/\mathfrak{d}_n \otimes_{\mathfrak{o}} \mathrm{Fil}^1H^1_{dR}(\mathcal{Y}_H(N)/\mathfrak{o}) \right) \ar[r, "(1/p^{n-m}tr_{n,m})_{n \geq m}"]& \mathfrak{o}_m \otimes_{\mathfrak{o}} \mathrm{Fil}^1H^1_{dR}(\mathcal{Y}_H(N)/\mathfrak{o}) \ar[u, "\cup \frac{1}{p^m}\log \chi_{cyc}" right]
\end{tikzcd}
$$
commutes up to $p^c$-torsion.
\end{cor}

\subsubsection{} Let
\begin{equation} \label{exp*}
\mrm{exp}^*: H^1\left(K_m, H^1_{\text{\'et}}(\mathcal{Y}_H(N) \otimes_{\mathfrak{o}} \overline{K}, \Q_p(1) \right) \rightarrow K_m \otimes_{\mathfrak{o}} \mrm{Fil}^1 H^1_{dR}(\mathcal{Y}_H(N))
\end{equation}
be the dual exponential map \cite{bloch-kato}, \cite{kato91} 1.2.4. To recall its definition, we need to introduce the ring $B_{dR}^+$ defined by Fontaine, a topological discrete valuation ring which is a $\Q_p$-algebra endowed with a continuous action of $\mrm{Gal}(\overline{\Q}_p/\Q_p)$. The map $\mrm{exp^*}$ is defined as the composite of the canonical map
$$
H^1\left(K_m, H^1_{\text{\'et}}(\mathcal{Y}_H(N) \otimes_{\mathfrak{o}} \overline{K}, \Q_p(1) )\right) \rightarrow H^1\left(K_m, B_{dR}^+ \otimes_{\Q_p} H^1_{\text{\'et}}(\mathcal{Y}_H(N) \otimes_{\mathfrak{o}} \overline{K}, \Q_p(1))  \right),
$$
and of the inverse of the cup-product with the logarithm of the cyclotomic character, which is an isomorphism
$$
H^1(K_m, B_{dR}^+ \otimes_{\Q_p} H^1_{\text{\'et}}(\mathcal{Y}_H(N) \otimes_{\mathfrak{o}} \overline{K}, \Q_p(1))) \overset{\sim}{\rightarrow} H^0(K_m, B_{dR}^+ \otimes_{\Q_p} H^1_{\text{\'et}}(\mathcal{Y}_H(N) \otimes_{\mathfrak{o}} \overline{K}, \Q_p(1)))
$$
and of the canonical isomorphisms
$$
H^0(K_m, B_{dR}^+ \otimes_{\Q_p} H^1_{\text{\'et}}(\mathcal{Y}_H(N) \otimes_{\mathfrak{o}} \overline{K}, \Q_p(1)))=K_m \otimes_{\Q_p} \mrm{Fil}^0 D_{dR}(H^1_{\text{\'et}}(\mathcal{Y}_H(N) \otimes_{\mathfrak{o}} \overline{K}, \Q_p(1)))$$
$$
= K_m \otimes_{\mathfrak{o}} \mrm{Fil}^1 H^1_{dR}(\mathcal{Y}_H(N)). 
$$
The homomorphisms defined above compose the following diagram.

\begin{cor} \label{main'12} The diagram 
$$
\begin{tikzcd}
\left( \underleftarrow{\lim}_n K_2\left(\mathcal{Y}_H(N)\otimes_{\mathfrak{o}} \mathfrak{o}_n \right) \otimes \mu_{p^n}^{-1} \right)^0 \ar[r, "(Tr_{n,m} \circ Ch_n)_{n \geq 1}"] \ar[d, "d\log" left]& H^2_{\text{\'et}} \left( \mathcal{Y}_H(N) \otimes_{\mathfrak{o}} K_m, \Z_p(1) \right)^0 \ar[d, "Hochschild-Serre"]\\
\underleftarrow{\lim}_n H^0\left(\overline{\mathcal{Y}_H(N)} \otimes_{\mathfrak{o}} \mathfrak{o}_n, \Omega^2_{\overline{\mathcal{Y}_H(N)} \otimes_{\mathfrak{o}} \mathfrak{o}_n}(\log) \right)(-1) \ar[d, "pr" left] & H^1\left(K_m, H^1_{\text{\'et}}(\mathcal{Y}_H(N) \otimes_{\mathfrak{o}} \overline{K}, \Z_p(1)) \right) \ar[d]\\
\underleftarrow{\lim}_n \left( \Omega^1_{\mathfrak{o}_n/\mathfrak{o}}(-1) \otimes_{\mathfrak{o}} \mrm{Fil}^1 H^1_{dR}(\mathcal{Y}_H(N)) \right) \ar[d, "d \log \zeta_{p^n} \otimes \langle \zeta_{p^n} \rangle^{-1} \mapsto 1" left] & H^1\left(K_m, H^1_{\text{\'et}}(\mathcal{Y}_H(N) \otimes_{\mathfrak{o}} \overline{K}, \Q_p(1)) \right) \ar[d, "\mrm{exp}^*"]\\
\underleftarrow{\lim}_n \left( \mathfrak{o}_n/\mathfrak{d}_n \otimes_{\mathfrak{o}} \mrm{Fil}^1 H^1_{dR}(\mathcal{Y}_H(N)) \right) \ar[r, "(1/p^n tr_{n,m})_{n \geq m}"] & K_m \otimes_{\mathfrak{o}} \mrm{Fil}^1 H^1_{dR}(\mathcal{Y}_H(N)).
\end{tikzcd}
$$
is commutative.
\end{cor}

\begin{rems} The statement of the commutativity of the diagram of Corollary \ref{main'12}, where the product of modular curves $\mathcal{Y}_H(N)=\mathcal{Y}(N) \times_{\mathfrak{o}} \mathcal{Y}(N)$ is replaced by a single modular curve $\mathcal{Y}(N)$ is \cite[Corollary 3.2.4]{scholl}.
\end{rems}

\begin{proof}[Proof of Corollary \ref{main'12} assuming Corollary \ref{corollary11}] The composition of the homomorphism $\pi_1: H^1_{\text{\'et}}\left( \mathcal{Y}_H(N) \otimes {\overline{K}}, \Z_p(1)\right) \rightarrow \mrm{Fil}^1 H^1_{dR}(\mathcal{Y}_H(N)) \otimes_{\mathfrak{o}} \widehat{\overline{\mathfrak{o}}}$ and of the map $\mrm{Fil}^1 H^1_{dR}(\mathcal{Y}_H(N)) \otimes_{\mathfrak{o}} \widehat{\overline{\mathfrak{o}}} \rightarrow \mrm{Fil}^1 H^1_{dR}(\mathcal{Y}_H(N)) \otimes_{\mathfrak{o}} \widehat{\overline{K}}$ induced by the inclusion $\widehat{\overline{\mathfrak{o}}} \hookrightarrow \widehat{\overline{K}}$ is a homomorphism 
$$
H^1_{\text{\'et}}\left( \mathcal{Y}_H(N) \otimes_{\mathfrak{o}} {\overline{K}}, \Z_p(1)\right) \rightarrow \mrm{Fil}^1 H^1_{dR}(\mathcal{Y}_H(N)) \otimes_{\mathfrak{o}} \widehat{\overline{K}}.
$$ This homomorphisms extends by $\widehat{\overline{K}}$-linearity to a homomorphism
$$\pi_1': H^1_{\text{\'et}}\left( \mathcal{Y}_H(N) \otimes_{\mathfrak{o}} {\overline{K}}, \Z_p(1)\right)\otimes_{\Z_p} \widehat{\overline{K}}\rightarrow \mrm{Fil}^1 H^1_{dR}(\mathcal{Y}_H(N)) \otimes_{\mathfrak{o}} \widehat{\overline{K}}.$$ It follows from the discussion in Section 3.2 of \cite{scholl} that $\pi'_1$ is the composite of the Hodge-Tate isomorphism
$$
H^1_{\text{\'et}}\left( \mathcal{Y}_H(N) \otimes_{\mathfrak{o}} {\overline{K}}, \Q_p(1)\right) \otimes_{\Q_p} \widehat{\overline{K}} \overset{\sim}{\rightarrow} H^0(\overline{\mathcal{Y}_H(N)}, \Omega^1_{\overline{\mathcal{Y}_H(N)}/\mathfrak{o}}(\log)) \otimes_{\mathfrak{o}} \widehat{\overline{K}} \oplus H^1(\overline{\mathcal{Y}_H(N)}, \mathcal{O}) \otimes_{\mathfrak{o}} \widehat{\overline{K}}(1)
$$
and of the canonical projection
$$
H^0(\overline{\mathcal{Y}_H(N)}, \Omega^1_{\overline{\mathcal{Y}_H(N)}/\mathfrak{o}}(\log)) \otimes_{\mathfrak{o}} \widehat{\overline{K}} \oplus H^1(\overline{\mathcal{Y}_H(N)}, \mathcal{O}) \otimes_{\mathfrak{o}} \widehat{\overline{K}}(1)
$$
$$
 \rightarrow H^0(\overline{\mathcal{Y}_H(N)}, \Omega^1_{\overline{\mathcal{Y}_H(N)}/\mathfrak{o}}(\log)) \otimes_{\mathfrak{o}} \widehat{\overline{K}}.
$$
As a consequence, by definition of the homorphism $\mrm{exp}^*$, the diagram
$$
\begin{tikzcd}
H^1(K_m, H^1(\mathcal{Y}_H(N) \otimes_{\mathfrak{o}} \overline{K}, \Q_p(1))) \ar[r, "\pi_1"] \ar[d, "\mrm{exp}^*"]& H^1(K_m, \widehat{\overline{K}} \otimes_{\mathfrak{o}} \mrm{Fil}^1H^1_{dR}(\mathcal{Y}_H(N)/\mathfrak{o})) \ar[d, equal]\\
K_m \otimes_{\mathfrak{o}} \mrm{Fil}^1H^1_{dR}(\mathcal{Y}_H(N)/\mathfrak{o}) \ar[r, "\cup \log \chi_{cyc}"] & H^1(K_m, \widehat{\overline{K}}) \otimes_{\mathfrak{o}} \mrm{Fil}^1H^1_{dR}(\mathcal{Y}_H(N)/\mathfrak{o})
\end{tikzcd}
$$
is commutative. Hence Corollary \ref{corollary11} implies Corollary \ref{main12}.
\end{proof}

\begin{pro} \label{homologicaltriviality} The natural inclusion
$$
\left( \underleftarrow{\lim}_n K_2\left(\mathcal{Y}_H(N) \otimes_{\mathfrak{o}} \mathfrak{o}_n \right) \otimes \mu_{p^n}^{-1} \right)^0 \subset \underleftarrow{\lim}_n K_2\left(\mathcal{Y}_H(N) \otimes_{\mathfrak{o}} \mathfrak{o}_n \right) \otimes \mu_{p^n}^{-1},
$$
where $\left( \underleftarrow{\lim}_n K_2\left(\mathcal{Y}_H(N) \otimes_{\mathfrak{o}} \mathfrak{o}_n \right) \otimes \mu_{p^n}^{-1} \right)^0$ is defined in section \ref{K20} above, is an equality.
\end{pro}

To prove this result, we will use the following Lemma. Even if its proof is standard, we would like to recall it for the convenience of the reader.

\begin{lem} \label{H0Iw}
The inverse limit $\underleftarrow{\lim}_m H^0(K_m, H^2_{\text{\'et}}(\mathcal{Y}_H(N) \otimes_{\mathfrak{o}} \overline{K}, \Z_p(1)))
,$ where the transition homomorphisms are the corestriction homomorphisms, vanishes.
\end{lem}

\begin{proof} The $\Z_p$-module $H^2_{\text{\'et}}(\mathcal{Y}_H(N) \otimes_{\mathfrak{o}} \overline{K}, \Z_p(1))$ is finitely generated. Hence, the increasing sequence of sub-modules
$$
H^0(K_1, H^2_{\text{\'et}}(\mathcal{Y}_H(N) \otimes_{\mathfrak{o}} \overline{K}, \Z_p(1))) \subset H^0(K_2, H^2_{\text{\'et}}(\mathcal{Y}_H(N) \otimes_{\mathfrak{o}} \overline{K}, \Z_p(1))) \subset \ldots
$$
is stationary. In other words, there exists an integer $M \geq 1$ such that for any $m \geq M$, we have 
$
H^0(K_m, H^2_{\text{\'et}}(\mathcal{Y}_H(N) \otimes_{\mathfrak{o}} \overline{K}, \Z_p(1)))=H^0(K_M, H^2_{\text{\'et}}(\mathcal{Y}_H(N) \otimes_{\mathfrak{o}} \overline{K}, \Z_p(1)))$. As a consequence, we have
$$
\underleftarrow{\lim}_m H^0(K_m, H^2_{\text{\'et}}(\mathcal{Y}_H(N) \otimes_{\mathfrak{o}} \overline{K}, \Z_p(1))) = \bigcap_{m \geq M} \left( [K_m : K_M] H^2_{\text{\'et}}(\mathcal{Y}_H(N) \otimes_{\mathfrak{o}} \overline{K}, \Z_p(1)) \right).
$$
For any $m$, the field $K_m$ is $\Q_p(\zeta_{Np^m})$ and the $p$-adic valuation of $[\Q_p(\zeta_{Np^m}):\Q_p]$ is greater or equal to $m-1$. Hence, the limit when $m$ goes to infinity of the $p$-adic valuation of $[K_m : K_M]$ goes to $+\infty$. This implies the vanishing of $$\bigcap_{m \geq M} \left( [K_m : K_M] H^2_{\text{\'et}}(\mathcal{Y}_H(N) \otimes_{\mathfrak{o}} \overline{K}, \Z_p(1)) \right)$$ and completes the proof of the Lemma.
\end{proof}

\begin{proof}[Proof of Proposition \ref{homologicaltriviality}] We need to show that the composite of the map
$$
\underleftarrow{\lim}_n K_2\left(\mathcal{Y}_H(N) \otimes_{\mathfrak{o}} \mathfrak{o}_n \right) \otimes \mu_{p^n}^{-1} \rightarrow H^2_{\text{\'et}}(\mathcal{Y}_H(N) \otimes_{\mathfrak{o}} K_m, \Z_p(1))
$$
defined above and of the map
$$
H^2_{\text{\'et}}(\mathcal{Y}_H(N) \otimes_{\mathfrak{o}} K_m, \Z_p(1)) \rightarrow H^2_{\text{\'et}}(\mathcal{Y}_H(N) \otimes_{\mathfrak{o}} \overline{K}, \Z_p(1))
$$ 
induced by $\mathcal{Y}_H(N) \otimes_{\mathfrak{o}} \overline{K} \rightarrow \mathcal{Y}_H(N) \otimes_{\mathfrak{o}} K_m$ is the zero map. But the image of this composite lies in $\underleftarrow{\lim}_m H^0(K_m, H^2_{\text{\'et}}(\mathcal{Y}_H(N) \otimes_{\mathfrak{o}} \overline{K}, \Z_p(1)))$, which is zero according to Lemma \ref{H0Iw}.
\end{proof}

\begin{cor} \label{main12} The diagram 
$$
\begin{tikzcd}
 \underleftarrow{\lim}_n K_2\left(\mathcal{Y}_H(N)\otimes_{\mathfrak{o}} \mathfrak{o}_n \right) \otimes \mu_{p^n}^{-1}  \ar[r, "\iota_* \circ (Tr_{n,m} \circ Ch_n)_{n \geq 1}"] \ar[d, "d\log" left]& H^4_{\text{\'et}} \left( \mathcal{Y}_G(N) \otimes_{\mathfrak{o}} K_m, \Z_p(2) \right)^0 \ar[d, "Hochschild-Serre"]\\
\underleftarrow{\lim}_n H^0\left(\overline{\mathcal{Y}_H(N)} \otimes_{\mathfrak{o}} \mathfrak{o}_n, \Omega^2_{\overline{\mathcal{Y}_H(N)} \otimes \mathfrak{o}_n}(\log) \right)(-1) \ar[d, "pr" left] & H^1\left(K_m, H^3_{\text{\'et}}(\mathcal{Y}_G(N) \otimes_{\mathfrak{o}} \overline{K}, \Z_p(2)) \right) \ar[d]\\
\underleftarrow{\lim}_n \left( \Omega^1_{\mathfrak{o}_n/\mathfrak{o}}(-1) \otimes_{\mathfrak{o}} \mrm{Fil}^1 H^1_{dR}(\mathcal{Y}_H(N)) \right) \ar[d, "d \log \zeta_{p^n} \otimes \langle \zeta_{p^n} \rangle^{-1} \mapsto 1" left] & H^1\left(K_m, H^3_{\text{\'et}}(\mathcal{Y}_G(N) \otimes_{\mathfrak{o}} \overline{K}, \Q_p(2)) \right) \ar[d, "\mrm{exp}^*"]\\
\underleftarrow{\lim}_n \left( \mathfrak{o}_n/\mathfrak{d}_n \otimes_{\mathfrak{o}} \mrm{Fil}^1 H^1_{dR}(\mathcal{Y}_H(N)) \right) \ar[r, "\iota_* \circ (1/p^n tr_{n,m})_{n \geq m}"] & K_m \otimes_{\mathfrak{o}} \mrm{Fil}^2 H^3_{dR}(\mathcal{Y}_G(N)).
\end{tikzcd}
$$
is commutative.
\end{cor}

\begin{proof}
The statement follows from Corollary \ref{main'12} and from the commutativity of the diagram
$$
\begin{tikzcd}
H^2_{\text{\'et}} \left( \mathcal{Y}_H(N) \otimes_{\mathfrak{o}} K_m, \Z_p(1) \right)^0  \ar[r, "\iota_*"] \ar[d, "Hochschild-Serre" left]& H^4_{\text{\'et}} \left( \mathcal{Y}_G(N) \otimes_{\mathfrak{o}} K_m, \Z_p(2) \right)^0 \ar[d, "Hochschild-Serre"]\\
H^1\left(K_m, H^1_{\text{\'et}}(\mathcal{Y}_H(N) \otimes_{\mathfrak{o}} \overline{K}, \Z_p(1)) \right)  \ar[d] & H^1\left(K_m, H^3_{\text{\'et}}(\mathcal{Y}_G(N) \otimes_{\mathfrak{o}} \overline{K}, \Z_p(2)) \right) \ar[d]\\
H^1\left(K_m, H^1_{\text{\'et}}(\mathcal{Y}_H(N) \otimes_{\mathfrak{o}} \overline{K}, \Q_p(1) \right) \ar[d, "\mathrm{exp}^*" left] & H^1\left(K_m, H^3_{\text{\'et}}(\mathcal{Y}_G(N) \otimes_{\mathfrak{o}} \overline{K}, \Q_p(2)) \right) \ar[d, "\mrm{exp}^*"]\\
K_m \otimes_{\mathfrak{o}} \mrm{Fil}^1 H^1_{dR}(\mathcal{Y}_H(N))) \ar[r, "\iota_*"] & K_m \otimes_{\mathfrak{o}} \mrm{Fil}^2 H^3_{dR}(\mathcal{Y}_G(N))
\end{tikzcd}
$$
which follows from the functoriality of the Hochschild-Serre spectral sequence and of the dual exponential map.
\end{proof}

In the following, we give a proof of Theorem \ref{cdfinite}.

\subsection{K\"ahler differentials of $2$-dimensional local fields with imperfect residue field} \label{prelimkahler} In this section, we introduce notation and review some general facts about K\"ahler differentials of local fields with imperfect residue field. We call such fields big local fields. We follow the presentation of \cite[\S 3.4]{scholl}, which treats the case of an arbitrary big local field.\\

Let $L$ denote an extension of $\Q_p$ such that $L$ is complete with respect to a discrete valuation and its residue field $l$ satisfies $[l:l^p]=p^2$. Let $A$ denote the ring of integers of $L$. For any subring $R \subset A$, let 
$$
\widehat{\Omega}^1_{A/R}=\underleftarrow{\lim}\, {\Omega}^1_{A/R}/p^n{\Omega}^1_{A/R}.
$$
Let $K \subset L$ denote a finite extension of $\Q_p$ with ring of integers $\mathfrak{o}$.

\begin{pro} \label{structure-diff} The following statements hold.
\begin{itemize}
\item[\text{(i)}] The $A$-module $\widehat{\Omega}^1_{A/\mathfrak{o}}$ is finitely generated.
\item[\text{(ii)}] If $T_1$ and $T_2$ are elements whose image in $l$ form a $p$-basis, then the elements $d\log T_1$ and $d\log T_2$ form a basis of the $L$-vector space $\widehat{\Omega}^1_{A/\mathfrak{o}} \otimes_{A} L$.
\item[\text{(iii)}] Let $\varpi_K$ denote a uniformizer in $K$. If $\varpi_K$ is prime in $A$, then $\widehat{\Omega}^1_{A/\mathfrak{o}}$ is free over $A$.
\end{itemize}
\end{pro}

\begin{proof} This is \cite[Proposition 3.4.3]{scholl}.
\end{proof}

Fix an algebraic closure $\overline{L}$ of $L$ and let $\overline{A}$ denote the integral closure of $A$ in $L$. For any ring $B$ such that $A \subset B \subset \overline{A}$ and any subring $R \subset B$, we let
$$
\widehat{\Omega}^1_{B/R}=\underrightarrow{\lim}\, \widehat{\Omega}^1_{A'/A'\cap R}
$$
where the limit runs over all finite extensions $A'/A$ contained in $B$. Let ${}_{p^n} \widehat{\Omega}^1_{\overline{A}/A}$, resp. ${}_{p^n} \widehat{\Omega}^1_{\overline{A}/\mathfrak{o}}$, denote the $p^n$-torsion submodule of $\widehat{\Omega}^1_{\overline{A}/A}$, resp. $\widehat{\Omega}^1_{\overline{A}/\mathfrak{o}}$ and let $T_p(\widehat{\Omega}^1_{\overline{A}/A})$, resp. $T_p(\widehat{\Omega}^1_{\overline{A}/\mathfrak{o}})$, denote the inverse limit $\underleftarrow{\lim}\, {}_{p^n} \widehat{\Omega}^1_{\overline{A}/A}$, resp. $\underleftarrow{\lim}\, {}_{p^n} \widehat{\Omega}^1_{\overline{A}/\mathfrak{o}}$. 

\begin{pro} \label{ml} We have a short exact sequence of $\widehat{\overline{A}}$-modules
\begin{equation}
0 \rightarrow T_p(\widehat{\Omega}^1_{\overline{A}/\mathfrak{o}}) \rightarrow T_p(\widehat{\Omega}^1_{\overline{A}/A}) \rightarrow \widehat{\overline{A}} \otimes_{A} \widehat{\Omega}^1_{A/\mathfrak{o}} \rightarrow 0.
\end{equation} 
where $\widehat{\overline{A}}$ denotes the $p$-adic completion of $\overline{A}$.
\end{pro}

\begin{proof}
This is a particular case of \cite[(3.3.4)]{scholl}. 
\end{proof}

Let for any (not necessarily finite) extension $L'/L$ let $A' \subset L'$ denote the integral closure of $A$ in $L$. Then we define
\begin{equation} \label{connecting-big}
\delta_{L'/L}: \widehat{\Omega}^1_{A/\mathfrak{o}} \otimes_{A} \widehat{A'} \rightarrow H^1 \left( L', T_p\widehat{\Omega}^1_{\overline{\mathcal{O}}/\mathfrak{o}} \right)
\end{equation}
to be the connecting homomorphism for the long exact sequence in Galois cohomology corresponding to the short exact sequence of Proposition \ref{ml}. Let $\delta_L=\delta_{L/L}$. Then $\delta_{L'/L}$ and $\delta_{L'}$ compose the commutative diagram
$$
\begin{tikzcd}
\widehat{\Omega}^1_{A/\mathfrak{o}} \otimes_{A} \widehat{A'} \ar[d] \ar[rd, "\delta_{L'/L}"] \\
\widehat{\Omega}^1_{A'/\mathfrak{o}} \ar[r, "\delta_{L'}"] & H^1 \left( L', T_p\widehat{\Omega}^1_{\overline{\mathcal{O}}/\mathfrak{o}} \right)
\end{tikzcd}
$$
where the vertical homomorphism is the caonical one.

\subsection{Galois cohomology of K\"ahler differentials} 
Let $\eta$ denote the generic point of the special fibre of  $\overline{\mathcal{Y}_H(N)}/\mathfrak{o}$ and let $\mathcal{O}$ denote the $p$-adic completion of the local ring $\mathcal{O}_{\overline{\mathcal{Y}_H(N)},\eta}$. Write $\mathcal{K}$ for the field of fractions of $\mathcal{O}$. Then $\mathcal{K}$ is a big local field of dimension $2$. Write $\mathcal{K}_n=\mathcal{K}(\zeta_{p^n})$ for any integer $n \geq 0$. Let $\mathcal{O}_n$ denote the integral closure of $\mathcal{O}$ in $\mathcal{K}_n$. We write $\mathcal{K}_{\infty}$ for $\mathcal{K}\overline{K}$ and $\mathcal{O}_\infty$ for the valuation ring of $\mathcal{K}_\infty$.\\

Let $T_1, T_2 \in \mathcal{O}^\times$ be units whose image in the residue field form a $p$-basis. Let $\mathcal{M}$ denote the extension $\mathcal{K} \left( T_1^{p^{-\infty}}, T_2^{p^{-\infty}} \right)$, let $\mathcal{M}_{\infty}=\mathcal{M}\mathcal{K}_{\infty}$, let $\mathcal{B}$ denote the valuation ring of $\mathcal{M}$, let $\mathcal{B}_\infty$ denote the valuation ring of $\mathcal{M}_\infty$ and let $\widehat{\mathcal{B}_\infty}$ denote its completion. As a $\mathcal{O}_\infty$-module, the ring $\mathcal{B}_\infty$ is free with basis $\left\{ T_1^{a_1}T_2^{a_2}, a_1, a_2 \in \Q_p/\Z_p \right\}$. By Kummer theory, the natural homomorphism $\mrm{Gal}(\mathcal{M}_\infty/\mathcal{K}_\infty) \rightarrow \Z_p(1)^2$ defined by $\sigma \mapsto \left( \left(\sigma(T_1^{p^{-n}})/T_1^{p^{-n}}\right)_{n \geq 0},  \left(\sigma(T_2^{p^{-n}})/T_2^{p^{-n}}\right)_{n \geq 0}\right)$ is an isomorphism. We denote by $\sigma_{(1,0)} \in \mrm{Gal}(\mathcal{M}_\infty/\mathcal{K}_\infty)$ (resp. $\sigma_{(0,1)} \in \mrm{Gal}(\mathcal{M}_\infty/\mathcal{K}_\infty)$) the inverse image of $(1,0)$ (resp. $(0,1)$) by this isomorphism. This means that $\sigma_{(1,0)}(T_1^{p^{-n}})=\zeta_{p^n}T_1^{p^{-n}}$ and $\sigma_{(1,0)}(T_2^{p^{-n}})=T_2^{p^{-n}}$ and that $\sigma_{(0,1)}(T_1^{p^{-n}})=T_1^{p^{-n}}$ and $\sigma_{(0,1)}(T_2^{p^{-n}})=\zeta_{p^n}T_2^{p^{-n}}$ for any $n \geq 0$.

\begin{lem} \label{colmez} 
For $a_1 ,a_2  \in \Q_p/\Z_p$, we have an isomorphism
$$
H^1 \left( \mathcal{M}_\infty/\mathcal{K}_\infty, \widehat{\mathcal{O}_\infty}T_1^{a_1}T_2^{a_2} \right)
\simeq \dfrac{\left\{ (x,y) \in \widehat{\mathcal{O}_\infty}^2, (\zeta_{p^{n_2}}^{r_2}-1)x=(\zeta_{p^{n_1}}^{r_1}-1)y\right\}}{\left\{ \left( (\zeta_{p^{n_1}}^{r_1}-1)u, (\zeta_{p^{n_2}}^{r_2}-1)u\right), u \in \widehat{\mathcal{O}_\infty} \right\}}
$$
where $r_i$ and $n_i$ are given by $a_i=r_i/p^{n_i}$. 
\end{lem}

\begin{proof} 
Let $Z^1 \left( \mathcal{M}_\infty/\mathcal{K}_\infty, \widehat{\mathcal{O}_\infty}T_1^{a_1}T_2^{a_2} \right)$ 
be the group of continuous $1$-cocycles. By evaluating $1$-cocycles at $\sigma_{(1,0)}, \sigma_{(0,1)}$, we have 
$$
Z^1 \left( \mathcal{M}_\infty/\mathcal{K}_\infty, \widehat{\mathcal{O}_\infty}T_1^{a_1}T_2^{a_2} \right) 
\longrightarrow (\widehat{\mathcal{O}_\infty} T_1^{a_1}T_2^{a_2})^2 , \ 
z\mapsto \left(z\left(\sigma_{(1,0)}\right), z\left(\sigma_{(0,1)} \right) \right) 
$$ 
We claim that this evaluation map induces an isomorphism 
\begin{equation}\label{equaition:isomorphism_for_Z1}
Z^1 \left( \mathcal{M}_\infty/\mathcal{K}_\infty , \widehat{\mathcal{O}_\infty}T_1^{a_1}T_2^{a_2} \right) 
\overset{\sim}{\longrightarrow} 
\left\{ (x,y) \in \widehat{\mathcal{O}_\infty}^2, (\zeta_{p^{n_2}}^{r_2}-1)x=(\zeta_{p^{n_1}}^{r_1}-1)y\right\} . 
\end{equation} 
Let us denote the image of the above evaluation as 
\begin{equation}\label{equation:definition_of_x_and_y}
\left(z\left(\sigma_{(1,0)}\right), z\left(\sigma_{(0,1)} \right) \right) = 
\left(x   T_1^{a_1}T_2^{a_2}, y  T_1^{a_1}T_2^{a_2}\right).
\end{equation}
By the cocycle condition, we have 
$$
z\left(\sigma_{(1,0)}\sigma_{(0,1)}\right) 
= {\sigma_{(0,1)}}z\left(\sigma_{(1,0)}\right) + z\left(\sigma_{(0,1)}\right) 
= z\left(\sigma_{(1,0)}\right) + 
{\sigma_{(1,0)}}z\left(\sigma_{(0,1)}\right), 
$$
which implies that $x$ and $y$ given in \eqref{equation:definition_of_x_and_y} satisfy the relation 
$(\zeta_{p^{n_2}}^{r_2}-1)x=(\zeta_{p^{n_1}}^{r_1}-1)y$. The evaluation at $\sigma_{(1,0)}, \sigma_{(0,1)}$ 
defines a homomorphism of abelian groups 
\begin{equation*}
Z^1 \left( \mathcal{M}_\infty/\mathcal{K}_\infty , \widehat{\mathcal{O}_\infty}T_1^{a_1}T_2^{a_2} \right) 
\overset{\sim}{\longrightarrow} 
\left\{ (x,y) \in \widehat{\mathcal{O}_\infty}^2, (\zeta_{p^{n_2}}^{r_2}-1)x=(\zeta_{p^{n_1}}^{r_1}-1)y\right\} . 
\end{equation*} 
Since the subgroup generated by $\sigma_{(1,0)}, \sigma_{(0,1)}$ 
is dense in $\mathrm{Gal}( \mathcal{M}_\infty/\mathcal{K}_\infty)$, this homomorphism is injective. 
\par 
In order to prove that the homomorphism in question is surjective, we take 
any $(x,y) \in \widehat{\mathcal{O}_\infty}^2$ such that $(\zeta_{p^{n_2}}^{r_2}-1)x=(\zeta_{p^{n_1}}^{r_1}-1)y$ and  we try to show that 
there exists a unique continuous cocycle $z$ belonging to $Z^1 \left( \mathcal{M}_\infty/\mathcal{K}_\infty, \widehat{\mathcal{O}_\infty}T_1^{a_1}T_2^{a_2} \right)$ such that $z\left(\sigma_{(1,0)} \right)=x T_1^{a_1}T_2^{a_2}$ and $z \left( \sigma_{(0,1)} \right)=y T_1^{a_1}T_2^{a_2}$. Let $z$ be the continuous map $\mrm{Gal}(\mathcal{M}_\infty/\mathcal{K}_\infty) \rightarrow \widehat{\mathcal{O}_\infty}T_1^{a_1}T_2^{a_2}$ such that for any $n, m \in \Z$ we have
\begin{eqnarray*}
z \left(\sigma_{(1,0)}^n \right) &=& \left\{
    \begin{array}{ll}
        \sum_{k=0}^{n-1}\zeta_{p^{n_1}}^{k r_1} x T_1^{a_1}T_2^{a_2}  & \mbox{if } n \geq 0 \\
\\
         \sum_{k=1}^{-n}\zeta_{p^{n_1}}^{-k r_1} x T_1^{a_1}T_2^{a_2} & \mbox{if } n<0 \\
    \end{array}
\right. \\
&\!& \\
z \left(\sigma_{(0,1)}^{m} \right) &=& \left\{
    \begin{array}{ll}
        \sum_{k=0}^{m-1}\zeta_{p^{n_2}}^{k r_2} y T_1^{a_1}T_2^{a_2} & \mbox{if } m \geq 0 \\
\\
         \sum_{k=1}^{-m}\zeta_{p^{n_2}}^{-k r_2} y T_1^{a_1}T_2^{a_2} & \mbox{if } m <0 \\
    \end{array}
\right.\\
&\!& \\
z \left(\sigma_{(1,0)}^n\sigma_{(0,1)}^m \right) &=& z \left( \sigma_{(1,0)}^n \right)+\zeta_{p^{n_1}}^{n r_1} z \left( \sigma_{(0,1)}^m \right).
\end{eqnarray*}
To prove that $z$ is a cocycle, it is enough to prove that it satisfies the cocycle condition on the dense subgroup of $\mrm{Gal}(\mathcal{M}_\infty/\mathcal{K}_\infty)$ generated by $\sigma_{(1,0)}$ and $\sigma_{(0,1)}$. If $n, m \geq 0$ we have
\begin{eqnarray*}
z \left( \sigma_{(0,1)}^m  \sigma_{(1,0)}^n \right) &=& z \left(\sigma_{(1,0)}^n\sigma_{(0,1)}^m \right)\\
 &=& z \left( \sigma_{(1,0)}^n \right)+\zeta_{p^{n_1}}^{n r_1} z \left( \sigma_{(0,1)}^m \right)\\
 &=&  \left( \frac{\zeta_{p^{n_1}}^{n r_1}-1}{\zeta_{p^{n_1}}^{r_1}-1} x+\zeta_{p^{n_1}}^{n r_1} \frac{\zeta_{p^{n_2}}^{m r_2}-1}{\zeta_{p^{n_2}}^{r_2}-1} y \right)T_1^{a_1} T_2^{a_2}\\
&=& \left( \frac{\zeta_{p^{n_2}}^{m r_2}-1}{\zeta_{p^{n_2}}^{r_2}-1} y + \zeta_{p^{n_2}}^{m r_2} \frac{\zeta_{p^{n_1}}^{n r_1}-1}{\zeta_{p^{n_1}}^{r_1}-1} x \right) T_1^{a_1} T_2^{a_2}\\
&=& z \left( \sigma_{(0,1)}^m \right)+\zeta_{p^{n_2}}^{m r_2} z \left( \sigma_{(1,0)}^n \right)
\end{eqnarray*}
where the fourth equality follows from the equality $(\zeta_{p^{n_2}}^{r_2}-1)x=(\zeta_{p^{n_1}}^{r_1}-1)y$ by an easy calculation. The reader will prove similarly that 
$$
z \left( \sigma_{(0,1)}^m  \sigma_{(1,0)}^n \right)=z \left( \sigma_{(0,1)}^m \right)+ \zeta_{p^{n_2}}^{m r_2} z \left( \sigma_{(1,0)}^n \right)
$$
for any $n, m \in \Z$. Let $a, b, c, d \in \Z$. Write $\tau_1=\sigma_{(1,0)}^{a}\sigma_{(0,1)}^{b}$ and $\tau_2=\sigma_{(1,0)}^{c} \sigma_{(0,1)}^{d}$. By what we have just proved we have
\begin{eqnarray*}
z \left( \tau_1 \tau_2 \right) &=& z \left( \sigma_{(1,0)}^{a+c} \sigma_{(0,1)}^{b+d}\right)\\
&=& z \left( \sigma_{(1,0)}^{a} \right)+ \zeta_{p^{n_1}}^{a r_1}  z \left( \sigma_{(1,0)}^{c} \right)+\zeta_{p^{n_1}}^{(a+c)r_1} \left( z \left( \sigma_{(0,1)}^{b} \right) + \zeta_{p^{n_2}}^{b r_2} z \left( \sigma_{(0,1)}^d \right) \right)
\end{eqnarray*}
and by definition of $z$ we have
\begin{eqnarray*}
z \left( \tau_1 \right)+\tau_1 z \left( \tau_2 \right) &=&  z \left( \sigma_{(1,0)}^{a} \right)+ \zeta_{p^{n_1}}^{a r_1}  z \left( \sigma_{(0,1)}^{b} \right)+\zeta_{p^{n_1}}^{a r_1} \zeta_{p^{n_2}}^{b r_2} \left( z \left( \sigma_{(1,0)}^{c} \right)+ \zeta_{p^{n_1}}^{c r_1 }z \left( \sigma_{(0,1)}^{d} \right)\right).
\end{eqnarray*}
Hence
\begin{eqnarray*}
z \left( \tau_1 \tau_2 \right)- \left( z \left( \tau_1 \right)+\tau_1 z \left( \tau_2 \right) \right) &=&  \zeta_{p^{n_1}}^{a r_1} \left( (1-\zeta_{p^{n_2}}^{b r_2})z\left( \sigma_{(1,0)}^c \right)+ (\zeta_{p^{n_1}}^{c r_1}-1)  z \left( \sigma_{(0,1)}^{b} \right)\right)=0.
\end{eqnarray*}
where the last equality follows again from $(\zeta_{p^{n_2}}^{r_2}-1)x=(\zeta_{p^{n_1}}^{r_1}-1)y$ by an easy calculation. This proves that $z$ is a cocycle, this completes the proof of the surjectivity of \eqref{equaition:isomorphism_for_Z1}. 
\par 
Let $B^1 \left( \mathcal{M}_\infty/\mathcal{K}_\infty, \widehat{\mathcal{O}_\infty}T_1^{a_1}T_2^{a_2} \right)$ 
be the group of $1$-boundary elements. We claim that 
\begin{equation}\label{equaition:isomorphism_for_B1}
B^1 \left( \mathcal{M}_\infty/\mathcal{K}_\infty , \widehat{\mathcal{O}_\infty}T_1^{a_1}T_2^{a_2} \right) 
\overset{\sim}{\longrightarrow} 
\left( (\zeta_{p^{n_1}}^{r_1}-1)u, (\zeta_{p^{n_2}}^{r_2}-1)u\right), u \in \widehat{\mathcal{O}_\infty}  . 
\end{equation} 
By evaluating 
$w \in B^1 \left( \mathcal{M}_\infty/\mathcal{K}_\infty, \widehat{\mathcal{O}_\infty}T_1^{a_1}T_2^{a_2} \right)$ at 
$\sigma_{(1,0)}, \sigma_{(0,1)}$, 
we have $w\left( \sigma_{(1,0)} \right)=(\zeta_{p^{n_1}}^{r_1}-1) u$ and $w\left( \sigma_{(0,1)} \right)=(\zeta_{p^{n_2}}^{r_2}-1) u$ for some $u \in \widehat{\mathcal{O}_\infty}$. It is straightforward to check that this gives an isomorphism \eqref{equaition:isomorphism_for_B1}. 
This concludes the proof.
\end{proof}

The proof of the following Lemma follows from a statement of \cite[\S 4. (c)]{faltings}. As this statement is not proved, we give a proof below.

\begin{lem} \label{faltings} The morphism $a: \widehat{\mathcal{O}_\infty}^2 \rightarrow H^1 \left( \mathcal{M}_\infty/\mathcal{K}_\infty, \widehat{\mathcal{B}_\infty}(1) \right)$ defined as
$$
\widehat{\mathcal{O}_\infty}^2 \simeq \mrm{Hom}\left( \Z_p(1)^2, \widehat{\mathcal{O}_\infty}(1) \right) \simeq H^1 \left( \mathcal{M}_\infty/\mathcal{K}_\infty, \widehat{\mathcal{O}_\infty}(1) \right) \rightarrow H^1 \left( \mathcal{M}_\infty/\mathcal{K}_\infty, \widehat{\mathcal{B}_\infty}(1) \right)
$$ 
where the last morphism is induced by the inclusion $\widehat{\mathcal{O}_\infty}(1) \rightarrow \widehat{\mathcal{B}_\infty}(1)$, is injective and its cokernel is killed by $(\zeta_p-1)^2$.
\end{lem}

\begin{proof} The $\widehat{\mathcal{O}_\infty}$-module $\widehat{\mathcal{B}_\infty}$ is the topological direct sum of the free $\widehat{\mathcal{O}_\infty}$-modules $\widehat{\mathcal{O}_\infty}T_1^{a_1}T_2^{a_2}$ with $a_1, a_2 \in \Q_p/\Z_p$. The Galois group $\mrm{Gal}(\mathcal{M}_\infty/\mathcal{K}_\infty)$ acts trivially on $\widehat{\mathcal{O}_\infty}$. For any $a_1, a_2 \in \Q_p/\Z_p$, write $a_i=r_i/p^{n_i}$ with $r_i, n_i  \in \Z$, $n_i \geq 0$. Then the topological generator $\sigma_{(1,0)}$ of the first factor of $\mrm{Gal}(\mathcal{M}_\infty/\mathcal{K}_\infty)$, resp. $\sigma_{(0,1)}$ of the second factor of $\mrm{Gal}(\mathcal{M}_\infty/\mathcal{K}_\infty)$ acts by multiplication by $\zeta_{p^{n_1}}^{r_1}$, resp. $\zeta_{p^{n_2}}^{r_2}$, on $T_1^{a_1}T_2^{a_2}$. As a consequence $H^1 \left( \mathcal{M}_\infty/\mathcal{K}_\infty, \widehat{\mathcal{B}_\infty}(1) \right)$ is the topological direct sum of the $H^1 \left( \mathcal{M}_\infty/\mathcal{K}_\infty, \widehat{\mathcal{O}_\infty}T_1^{a_1}T_2^{a_2} \right)$ with $a_1, a_2 \in \Q_p/\Z_p$. We compute the $H^1 \left( \mathcal{M}_\infty/\mathcal{K}_\infty, \widehat{\mathcal{O}_\infty}T_1^{a_1}T_2^{a_2} \right)$ as follows.  According to Lemma \ref{colmez}, this group is isomorphic to 
$
\dfrac{\left\{ (x,y) \in \widehat{\mathcal{O}_\infty}^2, (\zeta_{p^{n_2}}^{r_2}-1)x=(\zeta_{p^{n_1}}^{r_1}-1)y\right\}}{\left\{ \left( (\zeta_{p^{n_1}}^{r_1}-1)z, (\zeta_{p^{n_2}}^{r_2}-1)z\right), z \in \widehat{\mathcal{O}_\infty} \right\}}
$
where we have written $a_i=r_i/p^{n_i}$. For $a_1=a_2=0$ this cohomology is isomorphic to $\widehat{\mathcal{O}_\infty}^2$ and it is clear that the inclusion 
$
\widehat{\mathcal{O}_\infty}^2 \hookrightarrow H^1 \left( \mathcal{M}_\infty/\mathcal{K}_\infty, \widehat{\mathcal{B}_\infty}(1) \right)
$
that we obtain is the one defined in the statement of the Lemma. Assume that $a_1=0$ and $a_2 \neq 0$. Then $$
H^1 \left( \mathcal{M}_\infty/\mathcal{K}_\infty, \widehat{\mathcal{O}_\infty}T_1^{a_1}T_2^{a_2} \right) \simeq \widehat{\mathcal{O}_\infty}/(\zeta_{p^{n_2}}^{r_2}-1)\widehat{\mathcal{O}_\infty}
$$
and as $\zeta_{p^{n_2}}^{r_2}-1$ divides $\zeta_p-1$ in $\widehat{\mathcal{O}_\infty}$, this cohomology group is killed by $\zeta_p-1$. Similarly, if $a_1 \neq 0$ and $a_2 = 0$, the corresponding cohomology group is killed by $\zeta_p-1$. Assume that $a_1 \neq 0$ and $a_2 \neq 0$. There exist $d_1, d_2 \in \widehat{\mathcal{O}_\infty}$ such that $(\zeta_{p^{n_1}}^{r_1}-1)d_1=(\zeta_{p^{n_2}}^{r_2}-1)d_2=\zeta_p-1$. Let $(x,y) \in \widehat{\mathcal{O}_\infty}^2$ representing an element of $
H^1 \left( \mathcal{M}_\infty/\mathcal{K}_\infty, \widehat{\mathcal{O}_\infty}T_1^{a_1}T_2^{a_2} \right)
$. Write $u=d_1 d_2(\zeta_{p^{n_2}}^{r_2}-1)x=d_1 d_2(\zeta_{p^{n_1}}^{r_1}-1)y$. Then 
$
(\zeta_p-1)^2 x=(\zeta_{p^{n_1}}^{r_1}-1)u= (\zeta_{p^{n_2}}^{r_2}-1) u=(\zeta_p-1)^2 y.
$
Hence $H^1 \left( \mathcal{M}_\infty/\mathcal{K}_\infty, \widehat{\mathcal{O}_\infty}T_1^{a_1}T_2^{a_2} \right)$ is killed by $(\zeta_p-1)^2$ and the proof is complete.
\end{proof}

\begin{lem} \label{killed1} The morphism of $\widehat{\mathcal{O}_\infty}$-modules $c: 
\widehat{\mathcal{O}_\infty}^2 \rightarrow \widehat{\Omega}^1_{\mathcal{O}/\mathfrak{o}} \otimes_{\mathcal{O}} \widehat{\mathcal{O}_\infty}
$
defined by $(1,0) \mapsto d\log T_1$ and $(0,1) \mapsto d\log T_2$ is injective and its cokernel is killed by a power of $p$.
\end{lem}

\begin{proof} According to Proposition \ref{structure-diff} and the structure theorem for finitely generated modules over a principal ideal domain, the morphism of $\mathcal{O}$-modules
$\mathcal{O}^2 \rightarrow \widehat{\Omega}^1_{\mathcal{O}/\mathfrak{o}}$ defined by $(1,0) \mapsto d\log T_1$ and $(0,1) \mapsto d\log T_2$ induces an isomorphism
$\widehat{\Omega}^1_{\mathcal{O}/\mathfrak{o}} \simeq \mathcal{O}^2 \oplus  \bigoplus_{k} \widehat{\mathcal{O}}/(\varpi)^{i_k}$ where $\varpi$ is a uniformizer of the discrete valuation ring $\mathcal{O}$. By tensoring over $\mathcal{O}$ with $\widehat{\mathcal{O}_\infty}$, which is flat over $\mathcal{O}$, we obtain the isomorphism $\widehat{\Omega}^1_{\mathcal{O}/\mathfrak{o}} \otimes_{\mathcal{O}} \widehat{\mathcal{O}_\infty} \simeq \widehat{\mathcal{O}_\infty}^2 \oplus \bigoplus_{k} \widehat{\mathcal{O}_\infty}/(\varpi)^{i_k}$. In fact, because $\mathcal{O}/\mathfrak{o}$ is formally smooth, any uniformizer $\varpi_K$ of $K$ remains a uniformizer in $\mathcal{O}$ (see \cite[(28.G)]{matsumura} and \cite[Theorem 62 and 82]{matsumura}). Hence, we can assume that $\varpi=\varpi_K$. As $K/\Q_p$ is a finite extension, we have $(p)=(\varpi_K)^e$ and the statement follows.
\end{proof}

\begin{lem} \label{scholl3.4.12} Let $\mathfrak{m}_\infty \subset \mathcal{O}_\infty$ denote the maximal ideal. For any integer $j \geq 0$, the kernel and the cokernel of the inflation morphism
$$
H^j\left(\mathcal{M}_\infty/\mathcal{K}_\infty, \widehat{\mathcal{B}_\infty}(1) \right)= H^j\left( \mathcal{M}_\infty/\mathcal{K}_\infty, H^0\left( \mathcal{M}_\infty, \widehat{\overline{\mathcal{O}}}(1) \right) \right) \rightarrow H^j\left( \mathcal{K}_\infty, \widehat{\overline{\mathcal{O}}}(1) \right)
$$ 
are killed by $\mathfrak{m}_\infty$.
\end{lem}

\begin{proof} The statement follows from the first part of the proof of \cite[Proposition 3.4.12]{scholl}. Note that this part of the proof holds for any big local field. 
\end{proof}

Let $\widehat{\mathfrak{a}}_{\overline{\mathcal{O}}/\mathfrak{o}}$ denote the fractional ideal $(\zeta_p-1)^{-1}\mathfrak{d}^{-1}\widehat{\overline{\mathcal{O}}} \subset \widehat{\overline{\mathcal{L}}}$, where $\mathfrak{d}$ is the different of $\mathfrak{o}/\Z_p$. As $\mathfrak{d}$ is a principal ideal (\cite[III \S 6 Corollaire 2]{serre}), we have a canonical isomorphism $\widehat{\overline{\mathcal{O}}}(1) \overset{\sim}{\rightarrow} \widehat{\mathfrak{a}}_{\overline{\mathcal{O}}/\mathfrak{o}}(1)$. Let $$
b: H^1\left(\mathcal{M}_\infty/\mathcal{K}_\infty, \widehat{\mathcal{B}_\infty}(1) \right) \rightarrow H^1\left( \mathcal{K}_\infty, T_p\widehat{\Omega}^1_{\overline{\mathcal{O}}/\mathfrak{o}} \right)
$$ 
denote the morphism defined as the composite of the inflation 
$$
H^1\left(\mathcal{M}_\infty/\mathcal{K}_\infty, \widehat{\mathcal{B}_\infty}(1) \right)= H^1\left( \mathcal{M}_\infty/\mathcal{K}_\infty, H^0\left( \mathcal{M}_\infty, \widehat{\overline{\mathcal{O}}}(1) \right) \right) \rightarrow H^1\left( \mathcal{K}_\infty, \widehat{\overline{\mathcal{O}}}(1) \right)
$$
and of the isomorphism induced by $\widehat{\overline{\mathcal{O}}}(1) \overset{\sim}{\rightarrow} \widehat{\mathfrak{a}}_{\overline{\mathcal{O}}/\mathfrak{o}}(1) \overset{\sim}{\rightarrow} T_p\widehat{\Omega}^1_{\overline{\mathcal{O}}/\mathfrak{o}}$ where $\widehat{\mathfrak{a}}_{\overline{\mathcal{O}}/\mathfrak{o}}(1) \overset{\sim}{\rightarrow} T_p\widehat{\Omega}^1_{\overline{\mathcal{O}}/\mathfrak{o}}$ is the isomorphism \cite[(3.4.7)]{scholl}. 

\begin{cor} \label{b} The kernel and the cokernel of $b$ are killed by $p$.
\end{cor}

\begin{proof} As $p \in \mathfrak{m}_\infty$, it follows from Lemma \ref{scholl3.4.12} that the kernel and the cokernel of $b$ are killed by $p$.
\end{proof}

\begin{lem} \label{cd-faltings} The diagram of $\widehat{\mathcal{O}_\infty}$-modules
$$
\begin{tikzcd}
\widehat{\mathcal{O}_\infty}^2 \ar[r, "a"] \ar[d, "c"]& H^1\left(\mathcal{M}_\infty/\mathcal{K}_\infty, \widehat{\mathcal{B}_\infty}(1) \right) \ar[d, "b"]\\
\widehat{\Omega}^1_{\mathcal{O}/\mathfrak{o}} \otimes_{\mathcal{O}} \widehat{\mathcal{O}_\infty} \ar[r, "\delta_{\mathcal{K}_\infty}"] & H^1\left( \mathcal{K}_\infty, T_p\widehat{\Omega}^1_{\overline{\mathcal{O}}/\mathfrak{o}} \right)
\end{tikzcd}
$$
is commutative.
\end{lem}

\begin{proof} Let us prove that $(1,0)$ is sent to the same element of $H^1\left( \mathcal{K}_\infty, T_p\widehat{\Omega}^1_{\overline{\mathcal{O}}/\mathfrak{o}} \right)$ by the two paths of the diagram. The image of $(1,0) \in \widehat{\mathcal{O}_\infty}^2$ by the left hand vertical mrophism is $d\log T_1 \otimes 1$ and $\delta_{\mathcal{K}_\infty}(d\log T_1 \otimes 1)$ is the cohomology class of the continuous cocycle $\mrm{Gal}(\overline{\mathcal{K}_\infty}/\mathcal{K}_\infty) \rightarrow T_p\widehat{\Omega}^1_{\overline{\mathcal{O}}/\mathfrak{o}}$ defined by $\sigma \mapsto \left(\sigma(d\log T_1^{p^{-n}})-d\log T_1^{p^{-n}} \right)_{n \geq 0}$. The image of $(1,0)$ by the upper horizontal morphism is the cohomology class of the unique continuous cocycle $\mrm{Gal}(\mathcal{M}_\infty/\mathcal{K}_\infty) \rightarrow \widehat{\mathcal{B}_\infty}(1)$ defined by $\sigma_1 \mapsto 1 \otimes (\zeta_{p^n})_{n \geq 0} \in \widehat{\mathcal{B}_\infty} \otimes \Z_p(1)$ and $\sigma_2 \mapsto 0$. The image of $1 \otimes (\zeta_{p^n})_{n \geq 0}$ by the morphism $\widehat{\overline{\mathcal{O}}}(1) \rightarrow T_p\widehat{\Omega}^1_{\overline{\mathcal{O}}/\mathfrak{o}}$ defined above is $(d\log \zeta_{p^n})_{n \geq 0}$. Hence the image of this cohomology class by the right hand vertical morphism is the cohomology class of the continuous cocycle $\mrm{Gal}(\overline{\mathcal{K}_\infty}/\mathcal{K}_\infty) \rightarrow T_p\widehat{\Omega}^1_{\overline{\mathcal{O}}/\mathfrak{o}}$ defined by $\sigma \mapsto \left(d\log \zeta_{p^n}^{\chi_1(\sigma)} \right)_{n \geq 0}$ where $\chi_1(\sigma) \in \Z_p$ is defined by $\sigma(T_1^{p^{-n}})/T_1^{p^{-n}}=\zeta_{p^n}^{\chi_1(\sigma)}$ for any $n \geq 0$. In particular $d\zeta_{p^n}^{\chi_1(\sigma)}=\sigma(d\log T_1^{p^{-n}})-d\log T_1^{p^{-n}}$ for any $n \geq 0$. Similarly, one proves that $(0,1)$ is sent to the same element of $H^1\left( \mathcal{K}_\infty, T_p\widehat{\Omega}^1_{\overline{\mathcal{O}}/\mathfrak{o}} \right)$ by the two paths of the diagram. Hence the diagram is commutative as claimed.
\end{proof}

The following statement is the generalization of \cite[Proposition 3.4.12]{scholl} to the big local field $\mathcal{K}$ of dimension $2$.

\begin{pro} \label{hyodo-tate} Let 
\begin{equation} \label{delta}
\delta_{\mathcal{K}_{\infty}/\mathcal{K}}: \widehat{\Omega}^1_{\mathcal{O}/\mathfrak{o}} \otimes_{\mathcal{O}} \widehat{\mathcal{O}_{\infty}} \rightarrow H^1\left( \mathcal{K}_{\infty}, T_p\widehat{\Omega}^1_{\overline{\mathcal{O}}/\mathfrak{o}}\right)
\end{equation}
be the homomorphism \eqref{connecting-big} for $L'=\mathcal{K}_\infty$ and $L=\mathcal{K}$. There exists an integer $e \geq 0$ such that the kernel and the cokernel of $\delta_{\mathcal{K}_{\infty}/\mathcal{K}}$ are killed by $p^e$.
\end{pro}

\begin{proof} Let us prove that the kernel of $\delta_{\mathcal{K}_\infty/\mathcal{K}}$ is killed by a power of $p$. Let $x \in \ker \delta_{\mathcal{K}_\infty/\mathcal{K}}$. According to Lemma \ref{cd-faltings}, the diagram
$$
\begin{tikzcd}
\widehat{\mathcal{O}_\infty}^2 \ar[r, "a"] \ar[d, "c"]& H^1\left(\mathcal{M}_\infty/\mathcal{K}_\infty, \widehat{\mathcal{B}_\infty}(1) \right) \ar[d, "b"]\\
\widehat{\Omega}^1_{\mathcal{O}/\mathfrak{o}} \otimes_{\mathcal{O}} \widehat{\mathcal{O}_\infty} \ar[r, "\delta_{\mathcal{K}_\infty/\mathcal{K}}"] & H^1\left( \mathcal{K}_\infty, T_p\widehat{\Omega}^1_{\overline{\mathcal{O}}/\mathfrak{o}} \right)
\end{tikzcd}
$$
is commutative. According to Lemma \ref{killed1}, there exist an integer $u \geq 0$ and an element $y \in \widehat{\mathcal{O}_\infty}^2$ such that $p^ux=c(y)$. Applying $\delta_{\mathcal{K}_\infty/\mathcal{K}}$ we find $(b \circ a)(y)=0$, i.e. $a(y) \in \ker b$. According to Corollary \ref{b} this implies that $a(y)$ is killed by $p$ so that $py \in \ker a$. Hence, according to Lemma \ref{faltings}, we have $(\zeta_p-1)^2 p y=0$. As a consequence $(\zeta_p-1)^2 p^{u+1}x=0$. But we have $\prod_{v \in (\Z/p\Z)^\times} (\zeta_p^v-1)=p$ and so $p^{u+3}x=0$. This proves the claim. The fact that the cokernel of $\delta_{\mathcal{K}_\infty/\mathcal{K}}$ is killed by a power of $p$ is proved similarly using the commutativity of the diagram above and Lemma \ref{faltings}, Lemma \ref{killed1} and Corollary \ref{b}.
\end{proof}

\subsection{Proof of Theorem \ref{cdfinite}} Theorem \ref{cdfinite} is proved by reducing to the setting of the previous section. Namely we will reduce the proof of Theorem \ref{cdfinite} to the one of the commutativity of a diagram involving the big local fields $\mathcal{K}_n$ of dimension $2$. We first need to introduce the maps composing this diagram.

\subsubsection{} Let $s_n: K_2\left( \mathcal{O}_n\right)  \rightarrow H^2(\mathcal{K}_n, \Z/p^n\Z(2))$
be the map defined as the composite of the map $K_2\left( \mathcal{O}_n\right) \rightarrow K_2\left( \mathcal{K}_n\right)$ induced by the inclusion $\mathcal{O}_n \subset \mathcal{K}_n$ and of the Galois symbol $K_2(\mathcal{K}_n) \rightarrow H^2(\mathcal{K}_n, \Z/p^n\Z(2))$ defined for example in \cite[4.2.3]{kato99}. Composing with the twisting
$
H^2(\mathcal{K}_n, \Z/p^n\Z(2)) \otimes \mu_{p^n}^{-1} \rightarrow H^2(\mathcal{K}_n, \Z/p^n\Z(1)),
$
we obtain the map
$K_2\left( \mathcal{O}_n\right) \otimes \mu_{p^n}^{-1} \rightarrow H^2(\mathcal{K}_n, \Z/p^n\Z(1))$
which restricts to a map
$$
Ch_n: \left( K_2\left( \mathcal{O}_n \right) \otimes \mu_{p^n}^{-1} \right)^0 \rightarrow H^2(\mathcal{K}_n, \Z/p^n\Z(1))^0
$$
where
\begin{eqnarray*}
H^2(\mathcal{K}_n, \Z/p^n\Z(1))^0 &=& \ker\left( H^2(\mathcal{K}_n, \Z/p^n\Z(1)) \rightarrow H^2(\mathcal{K}_\infty, \Z/p^n\Z(1))\right),\\
\left( K_2(\mathcal{O}_n) \otimes \mu_{p^n}^{-1} \right)^0 &=& \ker\left( K_2(\mathcal{O}_n) \otimes \mu_{p^n}^{-1}  \rightarrow H^2(\mathcal{K}_n, \Z/p^n\Z(1)) \rightarrow H^2(\mathcal{K}_\infty, \Z/p^n\Z(1))\right).
\end{eqnarray*}

\subsubsection{} We denote by 
$$
d\log:  K_2(\mathcal{O}_n) \otimes \mu_{p^n}^{-1} \rightarrow \Omega^2_{\mathcal{O}_n/\mathfrak{o}} \otimes \mu_{p^n}^{-1}=\Omega^2_{\mathcal{O}_n/\mathfrak{o}}(-1)
$$
the map induced by $d\log_{ \mathcal{O}_n}: K_2\left( \mathcal{O}_n\right) \rightarrow \Omega^2_{\mathcal{O}_n/\Z}$ (see \cite[2.1]{scholl}).

\subsubsection{} 
\begin{lem} \label{dec-df} We have a canonical isomorphism of $\mathcal{O}_{n}$-modules
$$
\widehat{\Omega}^2_{\mathcal{O}_n/\mathfrak{o}} \simeq  \left( \widehat{\Omega}^1_{\mathcal{O}/\mathfrak{o}} \otimes_{\mathfrak{o}} \Omega^1_{\mathfrak{o}_n/\mathfrak{o}}\right) \oplus \left(\widehat{\Omega}^2_{\mathcal{O}/\mathfrak{o}} \otimes_{\mathfrak{o}} \mathfrak{o}_n \right)
$$
where the right hand side is regarded as a $\mathcal{O}_n$-module via the equality $\mathcal{O}_n=\mathcal{O} \otimes_{\mathfrak{o}} \mathfrak{o}_n$.
\end{lem}

\begin{proof}
As $N$ is prime to $p$, the algebra $\mathcal{O}/\mathfrak{o}$ is formally smooth. As a consequence $\mathcal{O}_n = \mathcal{O} \otimes_{\mathfrak{o}} \mathfrak{o}_n$ and $\mathcal{O}_n/\mathfrak{o}_n$ is formally smooth by \cite[Proposition 19.9.2 (iii)]{ega}. This implies that we have a canonical isomorphism of $\mathcal{O}_n$-modules
\begin{eqnarray*}
\Omega^1_{\mathcal{O}_n/\mathfrak{o}} &\simeq& (\Omega^1_{\mathfrak{o}_n/\mathfrak{o}} \otimes_{\mathfrak{o}_n} \mathcal{O}_n) \oplus \Omega^1_{\mathcal{O}_n/\mathfrak{o}_n}\\
&\simeq& (\Omega^1_{\mathfrak{o}_n/\mathfrak{o}} \otimes_{\mathfrak{o}_n} \mathcal{O}_n) \oplus (\Omega^1_{\mathcal{O}/\mathfrak{o}} \otimes_{\mathfrak{o}} \mathfrak{o}_n).
\end{eqnarray*}
by \cite[Th\'eor\`eme 20.5.7]{ega}. By taking the second exterior power, we have
$$
\Omega^2_{\mathcal{O}_n/\mathfrak{o}} \simeq \Omega^2_{\mathfrak{o}_n/\mathfrak{o}} \otimes_{\mathfrak{o}_n} \mathcal{O}_n \oplus \left( \Omega^1_{\mathcal{O}/\mathfrak{o}} \otimes_{\mathfrak{o}} \Omega^1_{\mathfrak{o}_n/\mathfrak{o}}\right) \oplus \left(\Omega^2_{\mathcal{O}/\mathfrak{o}} \otimes_{\mathfrak{o}} \mathfrak{o}_n \right)
$$
As the $\mathfrak{o}_n$-module $\Omega^1_{\mathfrak{o}_n/\mathfrak{o}} \simeq \mathfrak{o}_n/\mathfrak{d}_n$ is generated by one element, namely $d\log \zeta_{p^n}$, we have $\Omega^2_{\mathfrak{o}_n/\mathfrak{o}}=0$. This implies the statement by taking the $p$-adic completion.
\end{proof}

We denote by
$$
pr: \Omega^2_{\mathcal{O}_n/\mathfrak{o}}(-1) \rightarrow \Omega^1_{\mathfrak{o}_n/\mathfrak{o}}(-1) \otimes \widehat{\Omega}^1_{\mathcal{O}/\mathfrak{o}}
$$
the projection deduced form the statement of Lemma \ref{dec-df}.

\subsubsection{} We have $d\log: \mathcal{O}^\times_\infty/p^n \rightarrow \widehat{\Omega}^1_{\mathcal{O}_\infty/\mathfrak{o}}/p^n$ but $\widehat{\Omega}^1_{\mathcal{O}_\infty/\mathfrak{o}} \simeq \overline{\mathfrak{o}} \otimes_\mathfrak{o} \widehat{\Omega}^1_{\mathcal{O}/\mathfrak{o}} \oplus \Omega^1_{\overline{\mathfrak{o}}/\mathfrak{o}} \otimes_\mathfrak{o} \mathcal{O}$ and the second summand is divisible. In particular, we obtain the map
$$
d\log: \mathcal{O}^\times_\infty/p^n \rightarrow \overline{\mathfrak{o}}/p^n \otimes_\mathfrak{o}  \widehat{\Omega}^1_{\mathcal{O}/\mathfrak{o}}.
$$ 

\begin{thm} \label{main22} There exists an integer $e \geq 0$ such that the diagram 
$$
\begin{tikzcd}
\left( K_2(\mathcal{O}_n) \otimes \mu_{p^n}^{-1} \right)^0 \ar[r, "Ch_n"] \ar[d, "d\log" left]& H^2(\mathcal{K}_n, \Z/p^n\Z(1))^0 \ar[d, "Hochschild-Serre"]\\
\Omega^2_{\mathcal{O}_n/\mathfrak{o}}(-1) \ar[d, "pr" left] & H^1(K_n, H^1(\mathcal{K}_\infty,\Z/p^n\Z(1))) \ar[d, "\sim"]\\
\Omega^1_{\mathfrak{o}_n/\mathfrak{o}}(-1) \otimes_{\mathfrak{o}} \widehat{\Omega}^1_{\mathcal{O}/\mathfrak{o}} \ar[d, "d \log \zeta_{p^n} \otimes \langle \zeta_{p^n} \rangle^{-1} \mapsto 1" left] & H^1(K_n, \mathcal{O}_\infty^\times/p^n) \ar[d, "d\log"]\\
\mathfrak{o}_n/\mathfrak{d}_n \otimes_{\mathfrak{o}} \widehat{\Omega}^1_{\mathcal{O}/\mathfrak{o}} \ar[r, "\cup \frac{1}{p^n} \log \chi_{cyc}"]& H^1(K_n, \overline{\mathfrak{o}}/p^{n-1} \otimes_{\mathfrak{o}} \widehat{\Omega}^1_{\mathcal{O}/\mathfrak{o}}) 
\end{tikzcd}
$$
is commutative up to $p^e$-torsion.
\end{thm}

To prove Theorem \ref{main22} we will need several intermediary Lemmas.

\subsubsection{} We define the homomorphism 
\begin{equation*} \label{delta-gkn-n'}
\delta_{\mathcal{K}_n, n}: \widehat{\Omega}^1_{{\mathcal{O}}_{n}/\mathfrak{o}} \rightarrow H^1\left( \mathcal{K}_n, {}_{p^{n}}\widehat{\Omega}^1_{{\overline{\mathcal{O}}}/\mathfrak{o}} \right)
\end{equation*}
to be the composite of $\delta_{\mathcal{K}_n}:  \widehat{\Omega}^1_{{\mathcal{O}}_{n}/\mathfrak{o}} \rightarrow H^1\left( \mathcal{K}_n, T_p\left(\widehat{\Omega}^1_{{\overline{\mathcal{O}}}/\mathfrak{o}}\right) \right)$ which is defined at the end of section \ref{prelimkahler} and of the homomorphism induced by the projection $T_p\left(\widehat{\Omega}^1_{{\overline{\mathcal{O}}}/\mathfrak{o}}\right) \rightarrow {}_{p^{n}}\widehat{\Omega}^1_{{\overline{\mathcal{O}}}/\mathfrak{o}}$. Let $\left({\,\!}_{p^{n}}\widehat{\Omega}_{{\overline{\mathcal{O}}}/\mathfrak{o}}^{1} \right)^{\otimes 2}$ denote $\left( {}_{p^{n}}\widehat{\Omega}_{{\overline{\mathcal{O}}}/\mathfrak{o}}^{1} \right) \otimes_{{\overline{\mathcal{O}}}/p^{n}}\left(  {}_{p^{n}} \widehat{\Omega}_{{\overline{\mathcal{O}}}/\mathfrak{o}}^{1} \right)$. Let us also denote by 
\begin{equation} \label{wedge-delta}
\bigwedge^2 \delta_{\mathcal{K}_n,n}: \widehat{\Omega}^2_{\mathcal{O}_{n}/\mathfrak{o}} \rightarrow H^2\left(\mathcal{K}_n, \left({\,\!}_{p^{n}}\widehat{\Omega}_{{\overline{\mathcal{O}}}/\mathfrak{o}}^{1} \right)^{\otimes 2}\right)
\end{equation}
the homomorphism defined by $\left(\bigwedge^2 \delta_{\mathcal{K}_n,n} \right)(x \wedge y)=\delta_{\mathcal{K}_n,n}(x) \cup \delta_{\mathcal{K}_n,n}(y)$.\\

For any integer $n \geq 0$ we have the homomorphism $d\log_n: \Z/p^n\Z(1) \rightarrow {}_{p^n}\widehat{\Omega}_{{\overline{\mathcal{O}}}/\mathfrak{o}}^{1}$ and the homomorphism $(d\log_n)^{\otimes 2}:  \Z/p^n\Z(2) \rightarrow \left(\,\!_{p^n}\widehat{\Omega}_{{\overline{\mathcal{O}}}/\mathfrak{o}}^{1}\right)^{\otimes 2}$. 

\begin{lem} \label{cdul}
The diagram
$$
\begin{tikzcd}
\left( K_2(\mathcal{O}_n) \otimes \mu_{p^n}^{-1} \right)^0 \ar[r, "Ch_n"] \ar[d, "d\log" left] & H^2(\mathcal{K}_n, \Z/p^n\Z(1))^0 \ar[d, "d\log_n"]\\
\widehat{\Omega}^2_{\mathcal{O}_n/\mathfrak{o}}(-1) \ar[r, "\bigwedge^2 \delta_{\mathcal{K}_n,n}"] & H^2\left(\mathcal{K}_n, \left(\,_{p^n} \widehat{\Omega}^1_{\overline{\mathcal{O}}/\mathfrak{o}}\right)^{\otimes 2}(-1)\right)
\end{tikzcd}
$$
is commutative for every integer $n \geq 0$.
\end{lem}

\begin{proof} It follows from \cite[Lemma 3.5.2]{scholl} that the diagram
$$
\begin{tikzcd}
K_2\left( \mathcal{O}_n\right) \ar[r, "s_n"] \ar[d, "d\log" left] & H^2\left(\mathcal{K}_n, \Z/p^n\Z(2) \right) \ar[d, "d\log_n"]\\
\widehat{\Omega}^2_{\mathcal{O}_n/\mathfrak{o}} \ar[r, "\bigwedge^2 \delta_{\mathcal{K}_n,n}"] & H^2\left(\mathcal{K}_n, \left(\,\!_{p^n}\widehat{\Omega}_{{\overline{\mathcal{O}}}/\mathfrak{o}}^{1}\right)^{\otimes 2} \right)
\end{tikzcd}
$$
commutes for every integer $n \geq 0$. The statement follows by tensoring this diagram by $\mu_{p^n}^{-1}$.
\end{proof}
We define
$$
\left( \widehat{\Omega}^2_{\mathcal{O}_n/\mathfrak{o}}(-1) \right)^0=\ker \left(  \widehat{\Omega}^2_{\mathcal{O}_n/\mathfrak{o}}(-1) \rightarrow H^2\left(\mathcal{K}_n, \left(\,_{p^n}\widehat{\Omega}^1_{\overline{\mathcal{O}}/\mathfrak{o}}\right)^{\otimes 2}(-1)\right) \rightarrow H^2\left(\mathcal{K}_\infty, \left(\,_{p^n}\widehat{\Omega}^1_{\overline{\mathcal{O}}/\mathfrak{o}}\right)^{\otimes 2}(-1)\right)\right)
$$
where the first map is $\bigwedge^2 \delta_{\mathcal{K}_n,n}$ and the second map is the restriction. By Lemma \ref{cdul}, we have the commutative diagram
$$
\begin{tikzcd}
\left(  K_2(\mathcal{O}_n) \otimes \mu_{p^n}^{-1} \right)^0 \ar[r] \ar[d] & H^2(\mathcal{K}_n, \Z/p^n\Z(1))^0  \ar[d] \\
\left( \Omega^2_{\mathcal{O}_n/\mathfrak{o}}(-1) \right)^0 \ar[r]  & H^2\left(\mathcal{K}_n, \left(\,_{p^n}\widehat{\Omega}^1_{\overline{\mathcal{O}}/\mathfrak{o}}\right)^{\otimes 2}(-1)\right)^0.
\end{tikzcd}
$$

\subsubsection{} For any integer $n \geq 0$ we have the Hochschild-Serre spectral sequence
$$
E_2^{p,q}=H^p\left(K_n, H^q\left( \mathcal{K}_{\infty}, \left( \,\!_{p^{n}}\widehat{\Omega}_{\overline{\mathcal{O}}/\mathfrak{o}}^{1}\right)^{\otimes 2} \right) \right) \implies H^{p+q} \left( \mathcal{K}_n, \left( \,\!_{p^{n}}\widehat{\Omega}_{\overline{\mathcal{O}}/\mathfrak{o}}^{1} \right)^{\otimes 2} \right).
$$
As the cohomological dimension of the absolute Galois group of a finite extension of $\Q_p$ is $2$, we have $E_2^{3,0}=0$. Hence there is a natural projection
$$
HS_{n}: H^2\left( \mathcal{K}_n, \left( \,\!_{p^{n}}\widehat{\Omega}_{\overline{\mathcal{O}}/\mathfrak{o}}^{1} \right)^{\otimes 2} \right)^0 \rightarrow H^1\left(K_n, H^1\left( \mathcal{K}_{\infty}, \left( \,\!_{p^{n}}\widehat{\Omega}_{\overline{\mathcal{O}}/\mathfrak{o}}^{1} \right)^{\otimes 2}  \right) \right).
$$
Moreover, according to \cite[(3.3.6)]{scholl}, for any integer $n \geq 0$, we have a short exact sequence of $\mrm{Gal}(\overline{K}/K)$-modules
\begin{equation} \label{short-exact-diff1}
0 \rightarrow T_p\Omega^1_{\overline{\mathfrak{o}}/\mathfrak{o}} \rightarrow T_p\Omega^1_{\overline{\mathfrak{o}}/\mathfrak{o}_n} \rightarrow \overline{\mathfrak{o}} \otimes_{\mathfrak{o}_n} \Omega^1_{\mathfrak{o}_n/\mathfrak{o}} \rightarrow 0.
\end{equation}
Let
\begin{equation} \label{delta-pkn}
\delta_{K_n}: \Omega^1_{\mathfrak{o}_n/\mathfrak{o}} \rightarrow H^0 \left(K_n, \overline{\mathfrak{o}} \otimes_{\mathfrak{o}_n} \Omega^1_{\mathfrak{o}_n/\mathfrak{o}}\right) \rightarrow H^1\left( K_n, T_p\Omega^1_{\overline{\mathfrak{o}}/\mathfrak{o}} \right)
\end{equation}
be defined as the composite of the natural inclusion and of the connecting homomorphism for the long exact sequence in Galois cohomology corresponding to the short exact sequence \ref{delta-pkn}. We let
\begin{equation} \label{delta-pkn-n'}
\delta_{K_n, n}: \widehat{\Omega}^1_{\mathfrak{o}_n/\mathfrak{o}} \rightarrow H^1\left( K_n, \!\,_{p^{n}}\widehat{\Omega}^1_{\overline{\mathfrak{o}}/\mathfrak{o}} \right)
\end{equation}
denote the composite of $\delta_{K_n}$ and of the homomorphism induced by the projection $T_p\widehat{\Omega}^1_{\overline{\mathfrak{o}}/\mathfrak{o}} \rightarrow \,\!_{p^{n}} \widehat{\Omega}^1_{\overline{\mathfrak{o}}/\mathfrak{o}}$. 

\begin{lem} \label{middle-rectangle}The diagram
$$
\begin{tikzcd}
\left( \Omega^2_{\mathcal{O}_n/\mathfrak{o}} \right)^0 \ar[r, "\bigwedge^2 \delta_{\mathcal{K}_n, n}"] \ar[d, "pr" left]& H^2\left(\mathcal{K}_n, \left(\,_{p^n}\widehat{\Omega}^1_{\overline{\mathcal{O}}/\mathfrak{o}}\right)^{\otimes 2}\right)^0 \ar[d, "HS_n"]\\
\Omega^1_{\mathfrak{o}_n/\mathfrak{o}} \otimes_{\mathfrak{o}} \widehat{\Omega}^1_{\mathcal{O}/\mathfrak{o}}  \arrow[d, "\delta_{K_n,n} \otimes \id" left]
& H^1\left(K_n, H^1\left( \mathcal{K}_{\infty}, \left( \,\!_{p^{n}}\widehat{\Omega}_{\overline{\mathcal{O}}/\mathfrak{o}}^{1} \right)^{\otimes 2}  \right) \right) \arrow[d, "\sim"] \\
H^1 \left( K_n, \,\!_{p^{n}}\widehat{\Omega}^1_{\overline{\mathfrak{o}}/\mathfrak{o}} \otimes_{\overline{\mathfrak{o}}} \widehat{\Omega}^1_{\mathcal{O}/\mathfrak{o}}\right) \arrow[d, "\id \otimes \delta_{\mathcal{K}_\infty,n}"  left] 
& H^1\left(K_n, H^1\left( \mathcal{K}_{\infty}, \,\!_{p^{n}} \widehat{\Omega}^1_{\overline{\mathfrak{o}}/\mathfrak{o}} \otimes_{\overline{\mathfrak{o}}}  \,\!_{p^{n}}\widehat{\Omega}_{\overline{\mathcal{O}}/\mathfrak{o}}^{1} \right)\right) \arrow[d, "\sim"]\\
H^1 \left( K_n, \,\!_{p^{n}}\widehat{\Omega}^1_{\overline{\mathfrak{o}}/\mathfrak{o}} \otimes_{\overline{\mathfrak{o}}} H^1\left( \mathcal{K}_\infty,\,\!_{p^{n}} \widehat{\Omega}^1_{\overline{\mathcal{O}}/\mathfrak{o}}\right)\right)\arrow[r, equal] & H^1 \left( K_n, \,\!_{p^{n}}\widehat{\Omega}^1_{\overline{\mathfrak{o}}/\mathfrak{o}} \otimes_{\overline{\mathfrak{o}}} H^1\left( \mathcal{K}_\infty,\,\!_{p^{n}} \widehat{\Omega}^1_{\overline{\mathcal{O}}/\mathfrak{o}}\right)\right),
\end{tikzcd}
$$
where the second right hand vertical arrow is induced by the isomorphism $
_{p^{n}} \widehat{\Omega}^1_{\overline{\mathcal{O}}/\mathfrak{o}} \simeq \,\!_{p^{n}} \widehat{\Omega}^1_{\overline{\mathfrak{o}}/\mathfrak{o}} \otimes_{\overline{\mathfrak{o}}} \overline{\mathcal{O}}
$ given by \cite[(3.4.8)]{scholl},
is commutative.
\end{lem}

\begin{proof} The statement of the Lemma is a direct consequence of \cite[Lemma 3.5.3 (ii)]{scholl}, which does not use that the big local field is of dimension $1$. 
\end{proof}

\begin{proof}[Proof of Theorem \ref{main22}]
We consider the diagram
$$
\begin{tikzcd}[column sep=tiny]
\left(  K_2(\mathcal{O}_n) \otimes \mu_{p^n}^{-1} \right)^0 \ar[r] \ar[d] & H^2(\mathcal{K}_n, \Z/p^n\Z(1))^0  \ar[d] \ar[r] & H^1(K_n, H^1(\mathcal{K}_\infty,\Z/p^n\Z(1))) \ar[d]\\
\left( \Omega^2_{\mathcal{O}_n/\mathfrak{o}}(-1) \right)^0 \ar[r] \ar[d] & H^2\left(\mathcal{K}_n, \left(\,_{p^n}\widehat{\Omega}^1_{\overline{\mathcal{O}}/\mathfrak{o}}\right)^{\otimes 2}(-1)\right)^0 \ar[r] & H^1\left(K_n, H^1\left(\mathcal{K}_\infty, \left( \,_{p^n} \widehat{\Omega}^1_{\overline{\mathcal{O}}/\mathfrak{o}}\right)^{\otimes 2}(-1)\right)\right)\\
\Omega^1_{\mathfrak{o}_n/\mathfrak{o}}(-1) \otimes \widehat{\Omega}^1_{\mathcal{O}/\mathfrak{o}} \ar[d] \ar[r] & H^1(K_n, \,_{p^n} \Omega^1_{\overline{\mathfrak{o}}/\mathfrak{o}}(-1) \otimes \widehat{\Omega}^1_{\mathcal{O}/\mathfrak{o}}) \ar[r, "(*)"] & H^1(K_n, T) \ar[u, equal]\\
\mathfrak{o}_n/\mathfrak{d}_n \otimes \widehat{\Omega}^1_{\mathcal{O}/\mathfrak{o}} \ar[r] & H^1(K_n, \mathfrak{a}/p^n \otimes \widehat{\Omega}^1_{\mathcal{O}/\mathfrak{o}}) \ar[u, "\sim"] & \ar[l] H^1(K_n, \mathcal{O}_\infty^\times/p^n) \arrow[bend right=90]{uuu},
\end{tikzcd}
$$
where $T$ denotes the Galois module $\,_{p^n} \Omega^1_{\overline{\mathfrak{o}}/\mathfrak{o}}(-1) \otimes H^1(\mathcal{K}_\infty, \,_{p^n} \widehat{\Omega}^1_{\overline{\mathcal{O}}/\mathfrak{o}})$. The left hand upper square commutes according to Lemma \ref{cdul}. The right hand upper square commutes by functoriality of the Hochschild-Serre spectral sequence. The middle rectangle commutes according to Lemma \ref{middle-rectangle}. The left hand lower square commutes according to \cite[Lemma 3.3.8]{scholl}. The remaining part of the diagram commutes by \cite[Lemma 3.5.4]{scholl}, as this statement does not use that the big local field is of dimension $1$. The statement of the Theorem thus follows from the fact that the kernel and cokernel of the map labeled $(*)$ are killed by $p^e$ for some integer $e$ according to Proposition \ref{hyodo-tate}.
\end{proof}

\begin{proof}[Proof of Theorem \ref{cdfinite}] The proof is exactly the same as the one of \cite[Theorem 3.2.3]{scholl}, using Theorem \ref{main22}. We give some explanations for the convenience of the reader. We have the homomorphism 
$$
\mrm{Fil}^1 H^1_{dR}(\mathcal{Y}_H(N))=H^0\left(\overline{\mathcal{Y}_H(N)}, \Omega^1_{\overline{\mathcal{Y}_H(N)}}(\log) \right) \rightarrow \widehat{\Omega}^1_{\mathcal{O}/\mathfrak{o}}
$$
induced by the inclusion $\eta \rightarrow \overline{\mathcal{Y}_H(N)}$ of the generic point $\eta$ of the special fiber of $ \overline{\mathcal{Y}_H(N)}$. As the fibers of $\overline{\mathcal{Y}_H(N)}/\mathfrak{o}$ are connected, this homomorphism is injective and its cokernel is torsion free. As a consequence the commutativity of the diagram of Theorem \ref{main22} up to $p^e$ torsion for some integer $e$ follows from the statement of Theorem \ref{main22} after localization to $\Spec \mathcal{O}$.
\end{proof}

\end{document}